\documentclass[letterpaper, 11pt,  reqno]{amsart}

\usepackage{amsmath,amssymb,amscd,amsthm,amsxtra, esint}
\usepackage{tikz} 
\usepackage{color}
\usepackage[implicit=true]{hyperref}

\usepackage{cases}

\usepackage[left=32mm, right=32mm, 
bottom=27mm]{geometry}

\allowdisplaybreaks[2]

\usepackage{color}

\definecolor{gr}{rgb}   {0.,   0.69,   0.23 }
\definecolor{bl}{rgb}   {0.,   0.5,   1. }
\definecolor{mg}{rgb}   {0.85,  0.,    0.85}
\definecolor{yl}{rgb}   {0.8,  0.7,   0.}
\definecolor{or}{rgb}  {0.7,0.2,0.2}

\usetikzlibrary{shapes.misc}
\usetikzlibrary{shapes.symbols}
\usetikzlibrary{shapes.geometric}

\tikzset{
	ddot/.style={circle,fill=white,draw=black,inner sep=0pt,minimum size=0.8mm},
	>=stealth,
	}

\tikzset{
	ddot2/.style={circle,fill=black,draw=black,inner sep=0pt,minimum size=0.8mm},
	>=stealth,
	}

\newtheorem{theorem}{Theorem} [section]

\newtheorem{lemma}[theorem]{Lemma}
\newtheorem{proposition}[theorem]{Proposition}
\newtheorem{remark}[theorem]{Remark}

\newtheorem*{acknowledgment}{Acknowledgments}

\DeclareMathOperator*{\supp}{supp}

\newcommand{\J}{\mathcal{J}}

\newcommand{\noi}{\noindent}
\newcommand{\Z}{\mathbb{Z}}
\newcommand{\R}{\mathbb{R}}

\newcommand{\T}{\mathbb{T}}

\let\Im=\undefined\DeclareMathOperator*{\Im}{Im}

\let\P= \undefined
\newcommand{\P}{\mathbf{P}}

\newcommand{\E}{\mathbb{E}}

\newcommand{\F}{\mathcal{F}}

\newcommand{\al}{\alpha}
\newcommand{\be}{\beta}
\newcommand{\dl}{\delta}

\newcommand{\nb}{\nabla}

\newcommand{\deff}{\stackrel{\textup{def}}{=}}

\newcommand{\Dl}{\Delta}
\newcommand{\eps}{\varepsilon}
\newcommand{\kk}{\kappa}
\newcommand{\g}{\gamma}
\newcommand{\G}{\Gamma}
\newcommand{\ld}{\lambda}

\newcommand{\s}{\sigma}
\newcommand{\Si}{\Sigma}
\newcommand{\ft}{\widehat}

\newcommand{\wt}{\widetilde}
\newcommand{\cj}{\overline}
\newcommand{\dx}{\partial_x}
\newcommand{\dt}{\partial_t}

\newcommand{\ta}{\theta}

\renewcommand{\l}{\ell}
\renewcommand{\o}{\omega}
\renewcommand{\O}{\Omega}

\newcommand{\les}{\lesssim}
\newcommand{\ges}{\gtrsim}

\newcommand{\jb}[1]
{\langle #1 \rangle}

\newcommand{\ind}{\mathbf 1}

\renewcommand{\S}{\mathcal{S}}

\newcommand{\too}{\longrightarrow}

\newcommand{\N}{\mathbb{N}}

\renewcommand{\H}{\mathcal{H}}

\newcommand{\U}{\Theta}

\numberwithin{equation}{section}
\numberwithin{theorem}{section}

\begin{document}
\baselineskip = 15pt

\title[2-$d$ hyperbolic stochastic  sine-Gordon equation]
{On the two-dimensional hyperbolic 
stochastic sine-Gordon equation}

\author[T.~Oh, T.~Robert, P.~Sosoe, and Y.~Wang]
{Tadahiro Oh, Tristan Robert, Philippe Sosoe, and Yuzhao Wang}

\address{
Tadahiro Oh, School of Mathematics\\
The University of Edinburgh\\
and The Maxwell Institute for the Mathematical Sciences\\
James Clerk Maxwell Building\\
The King's Buildings\\
Peter Guthrie Tait Road\\
Edinburgh\\ 
EH9 3FD\\
 United Kingdom}

\email{hiro.oh@ed.ac.uk}

\address{
Tristan Robert, School of Mathematics\\
The University of Edinburgh\\
and The Maxwell Institute for the Mathematical Sciences\\
James Clerk Maxwell Building\\
The King's Buildings\\
Peter Guthrie Tait Road\\
Edinburgh\\ 
EH9 3FD\\
 United Kingdom}

 \curraddr{Fakult\"at f\"ur Mathematik\\
Universit\"at Bielefeld\\
Postfach 10 01 31\\
33501 Bielefeld\\
Germany}

\email{trobert@math.uni-bielefeld.de}

\address{
Philippe Sosoe, Department of Mathematics\\ Cornell University, 584 Malott Hall, Ithaca\\
New York 14853, USA}
\email{psosoe@math.cornell.edu}

\address{
Yuzhao Wang\\
School of Mathematics, 
University of Birmingham, 
Watson Building, 
Edgbaston, 
Birmingham\\
B15 2TT, 
United Kingdom}

\email{y.wang.14@bham.ac.uk}



\begin{abstract}
We study the two-dimensional stochastic  sine-Gordon equation (SSG) 
in the hyperbolic setting.
In particular, by introducing a suitable time-dependent  renormalization
for the relevant imaginary  Gaussian  multiplicative  chaos, 
we prove local well-posedness of  SSG 
for {\it any} value of a parameter $\be^2 > 0$ in the nonlinearity.
This exhibits  sharp contrast with the parabolic case 
studied by Hairer and Shen (2016) and Chandra, Hairer, and Shen (2018),
 where the parameter is restricted to the  subcritical range:  $0 < \be^2 < 8 \pi$.
We also present a triviality result for the unrenormalized SSG.
\end{abstract}


\maketitle

\tableofcontents

\baselineskip = 14pt

\section{Introduction}
\subsection{Stochastic sine-Gordon equation}

We consider 
the following hyperbolic stochastic sine-Gordon equation   on $\T^2 = (\R/2\pi\Z)^2$
with an additive space-time white noise forcing:
\begin{align}
\begin{cases}
\dt^2 u + (1- \Dl)  u   +  \g \sin(\be u) = \xi\\
(u, \dt u) |_{t = 0} = (u_0, u_1) , 
\end{cases}
\qquad (t, x) \in \R_+\times\T^2,
\label{SSG}
\end{align}

\noi
where 
$\g$ and $\be$ are non-zero real numbers
and 
$\xi(t, x)$ denotes a (Gaussian) space-time white noise on $\R_+\times\T^2$.
In this paper, 
we are interested in studying the model 
\eqref{SSG}, 
where the linear dynamics is given by the wave equation.\footnote
{More precisely, by the Klein-Gordon equation.
In the following, we study \eqref{SSG}
with the Klein-Gordon operator $\dt^2  + (1- \Dl)$
to avoid a separate treatment at the zeroth frequency.
Note, however, that  the same results with inessential modifications also hold for 
 \eqref{SSG} with the wave operator $\dt^2- \Dl$.
 The same comment applies to \eqref{SNLW} (and \eqref{SNLW2}, respectively), 
 which we simply refer to as the stochastic nonlinear (and linear, respectively) wave equation
 in the following.}

The stochastic nonlinear wave equations (SNLW)
 have been studied extensively
in various settings; 
see \cite[Chapter 13]{DPZ14} for the references therein.
In recent years, there has been 
a rapid progress on the theoretical understanding 
of SNLW with singular stochastic forcing.
 In~\cite{GKO},  Gubinelli, Koch, 
and  the first author
studied the following stochastic nonlinear wave equations 
  with an additive space-time white noise on $\T^2$:
\begin{equation}
\label{SNLW}
 \dt^2 u + (1- \Dl)  u  + u^k =  \xi, 
\end{equation}

\noi
where $k\geq 2$ is an integer.
The main difficulty of this problem comes from the roughness
of the space-time white noise $\xi$,
which can already be seen at the linear level.
Let $\Psi$ denote 
 the stochastic convolution, 
solving the linear stochastic wave equation:
\begin{equation}
\label{SNLW2}
 \dt^2 \Psi + (1- \Dl)  \Psi   =  \xi.
\end{equation}

\noi
For the spatial dimension $d \geq 2$, 
the stochastic convolution  $\Psi$ 
is not a classical function but is merely a Schwartz distribution.
See Subsection~\ref{SUBSEC:1.2} and Lemma~\ref{LEM:psi} below.
This causes an issue in making sense of powers $\Psi^k$ and a fortiori of the full nonlinearity 
 $u^k$ in~\eqref{SNLW}, 
 necessitating 
 a renormalization of the nonlinearity. 
In \cite{GKO}, by introducing an appropriate {\it time-dependent}
renormalization, 
the authors proved local well-posedness of (a renormalized version of) \eqref{SNLW} on $\T^2$.
See \cite{OTh2, GKOT, GKO2, ORTz, OPT1, OOR}
for related works on nonlinear wave equations with singular stochastic forcing
and/or rough random initial data.
We point out that these works handle polynomial nonlinearities.

Our main goal in this paper is 
 to extend  the analysis to SNLW
 with a non-polynomial nonlinearity, 
 of which trigonometric functions
 are the simplest.
As in the case of a polynomial nonlinearity, 
a proper renormalization needs to be introduced to our problem.
This can be seen from the regularity of the stochastic convolution $\Psi$ as above.
More informally, we can see the necessity of a renormalization
from the fast oscillation of 
$\sin (\be u)$ due to the (expected) roughness of $u$, 
which makes the nonlinearity $ \sin (\be u)$ tend to $0$
in some limiting sense.  See Proposition \ref{PROP:triviality}.
In order to counterbalance such decay and have a non-trivial solution, we will take $\g \to \infty$ in~\eqref{SSG}.
See Subsection~\ref{SUBSEC:1.2}.

 The main new difficulty comes from the non-polynomial nature of the nonlinearity, 
 which makes the analysis of the relevant stochastic object
 particularly non-trivial (as compared to the previous work~\cite{GKO}).
 In particular, 
 after introducing a time-dependent renormalization, 
 we show that the regularity of the main stochastic terms
 depends on both the parameter $\be \in \R \setminus \{0\}$ and the time $t>0$.
See Proposition \ref{PROP:Ups}.
This is a striking difference from the polynomial case.

The sine nonlinearity in \eqref{SSG} is  closely related to models arising from both relativistic and quantum field theories \cite{PS, BEMS, Fro} and  has  attracted a lot of attention over the past years. 
Moreover, the following (deterministic) one-dimensional sine-Gordon equation:
\begin{align}
 \dt^2 u - \dx^2 u + \sin u = 0
\label{SG1}
\end{align} 

\noi
 is known to be completely integrable, to which  a vast literature was devoted. 
 Here, we simply mention the work by McKean 
 \cite{McKean81,McKean94}.
 See the references therein.
  In particular, in~\cite{McKean94}, 
McKean constructed an invariant Gibbs measure for \eqref{SG1}.

In the one-dimensional case,
the stochastic convolution $\Psi$ in \eqref{SNLW2} has (spatial) regularity $\frac 12 - \eps$, $\eps > 0$, 
and thus 
a straightforward computation yields local well-posedness of \eqref{SSG}.
The situation becomes much more delicate in the 
 two-dimensional setting.
In the parabolic setting,
Hairer-Shen \cite{HS}
and Chandra-Hairer-Shen \cite{CHS}
studied the following 
parabolic sine-Gordon model on $\T^2$:
 \begin{align}
\dt u -  \tfrac 12 \Dl  u   +  \g \sin(\be u) = \xi.
\label{SG2}
\end{align}

\noi
In particular, they observed that 
the difficulty of the problem depends sensitively on the value of $\be^2 > 0$.
At a heuristic level, 
this model is comparable 
to various models appearing in  constructive quantum field theory,
where the dimension $d$ can be expressed in terms of the parameter $\be$;
for example, the $\Phi^3_d$-model 
(and the $\Phi^4_d$-model, respectively) 
formally corresponds 
to \eqref{SG2} with $d = 2+ \frac{\be^2}{2\pi}$
(and  $d = 2+ \frac{\be^2}{4\pi}$, respectively).
In terms of the actual well-posedness theory,  
the Da Prato-Debussche trick \cite{DPD}
along with a standard Wick renormalization yields
local well-posedness of \eqref{SG2}
for  $0 < \be^2 <4\pi$.
It turns out that there is an infinite number of thresholds:
$\be^2 = \frac{j}{j+1}8\pi$, $j\in\N$, 
where one encounters new divergent stochastic objects, requiring further renormalizations.
By using the theory of regularity structures~\cite{H}, 
Hairer and Shen  proved local well-posedness up to the second threshold $\be^2<\frac{16\pi}{3}$
in the first work~\cite{HS}.
In the second work~\cite{CHS}, 
together with Chandra, 
they  pushed the local well-posedness theory to the entire subcritical regime $\be^2<8\pi$. 
When $\be^2 = 8 \pi$, 
the equation~\eqref{SG2}  is critical
and falls outside the scope of the current theory.
In fact, it is expected that, for $\be^2 \geq 8 \pi$, 
any reasonable approach would yield a trivial solution, 
solving the linear stochastic heat equation.
The equation \eqref{SG2} possesses
a formally invariant Gibbs measure, 
and thus it is expected that Bourgain's invariant measure
argument \cite{BO94, BO96} would allow one
to extend 
the local-in-time dynamics constructed in \cite{HS, CHS}
globally in time.
We mention a recent work~\cite{LRV2} on the construction of
a Gibbs measure for the sine-Gordon model
with a log-correlated Gaussian process
(but their result is restricted to $d = 1$).
We also mention  recent papers \cite{Ga, ST}
on dynamical problems with 
 an exponential nonlinearity in the two-dimensional setting;
in \cite{Ga}, Garban studied
the dynamical problem with an exponential nonlinearity 
in the parabolic setting, 
while in \cite{ST}, Sun and Tzvetkov
considered 
a dispersion generalized nonlinear wave equation 
with an exponential nonlinearity 
in the context of random data well-posedness theory.\footnote{See also a recent preprint
\cite{ORW}, where
we studied SNLW on $\T^2$ with the exponential nonlinearity.}

Our model \eqref{SSG} in the hyperbolic setting
is  sensitive to the value of $\be^2 >0$ as in the parabolic case.
Furthermore, due to the non-conservative nature of the problem,\footnote{We are
considering the problem without the damping term,
where there is no invariant measure for the linear dynamics.  See Remark \ref{REM:Gibbs} below.}
the renormalization we employ is time-dependent as in \cite{GKO}
and 
the difficulty also depends on the time parameter $t>0$.
See Proposition \ref{PROP:Ups}.
On the one hand, by taking $t>0$ small, 
we can make sure that the relevant stochastic object is not too rough, allowing
us to establish local well-posedness of \eqref{SSG}
for {\it any} value of $\be^2 > 0$
(Theorem~\ref{THM:main}).
On the other hand, even if we start the dynamics with small $\be^2 > 0$, 
our analysis, when compared to  the parabolic setting \cite{HS, CHS},  
formally indicates existence of 
an infinite number of thresholds $T_j = T_j(\be)$, now given in terms of time, 
\[ T_j = \frac{16j\pi}{(j+1)\be^2}, \qquad j \in \N, \]

\noi
where we encounter new divergent stochastic objects,  
requiring further renormalizations.
As in the parabolic case, the time $T_\infty = \frac{16 \pi}{ \be^2}$
corresponds to the critical value, after which 
we do not expect to be able to extend the dynamics.\footnote{Or perhaps, 
the dynamics may trivialize to the linear dynamics after the critical time $T_\infty$.}
It is quite intriguing that 
the singular nature of the problem \eqref{SSG}
depends sensitively on time
and gets worse over time, contrary to  the parabolic setting.

\subsection{Renormalization of the nonlinearity} 
\label{SUBSEC:1.2}

In order to explain the renormalization procedure, we
first consider the following regularized equation
for \eqref{SSG}:
\begin{align}
\dt^2 u_N + (1-  \Dl)  u_N   + \g \Im \big(e^{i\be u_N}\big)   = \P_N \xi, 
\label{SSGN}
\end{align}

\noi
where  
$\P_N$ is a smooth frequency projector
onto the (spatial) frequencies  $\{n\in\Z^2:|n|\leq N\}$,
associated with  a Fourier multiplier
\begin{align}
\chi_N(n) = \chi\big(N^{-1}n\big)
\label{chi}
\end{align}

\noi
for some fixed non-negative function $\chi \in C^\infty_c(\R^2)$
with $\supp \chi \subset \{\xi\in\R^2:|\xi|\leq 1\}$ and $\chi\equiv 1$ 
on $\{\xi\in\R^2:|\xi|\leq \tfrac12\}$.

We first  define the truncated stochastic convolution $\Psi_N = \P_N \Psi$, 
solving the truncated linear stochastic  wave equation:
\begin{align}
\dt^2 \Psi_N + (1-\Dl)\Psi_N = \P_N \xi
\label{PsiN0}
\end{align}

\noi
with the zero initial data.
With  $\jb{\,\cdot\,} = (1+|\cdot|^2)^\frac{1}{2}$, 
let $S(t)$ denote 
the linear wave propagator
\begin{equation}\label{S}
S(t) = \frac{\sin(t\jb{\nabla})}{\jb{\nabla}},
\end{equation} 

\noi
defined as a Fourier multiplier operator.
Namely, 
we set 
\[ 
S(t) f  = \sum_{n \in \Z^2 }
\frac{\sin(t\jb{n})}{\jb{n}}
\ft f (n) e_n, \]

\noi
where
$\ft{f}(n)$ is the Fourier coefficient of $f$
and
$e_n(x)= (2\pi)^{-1}e^{in\cdot x}$ as in \eqref{exp}.
Then, the truncated stochastic convolution  $\Psi_N$, 
solving \eqref{PsiN0}, is given by 
\begin{align} 
\Psi_N (t) 
 = \int_0^tS(t - t')\P_NdW(t'), 
\label{PsiN}
\end{align}

\noi
where  $W$ denotes a cylindrical Wiener process on $L^2(\T^2)$:
\begin{align}
W(t)\stackrel{\text{def}}{=} \sum_{n \in \Z^2 } B_n (t) e_n
\label{W1}
\end{align}

\noi
and  
$\{ B_n \}_{n \in \Z^2}$ 
is defined by 
$B_n(0) = 0$ and 
$B_n(t) = \jb{\xi, \ind_{[0, t]} \cdot e_n}_{ t, x}$.
Here, $\jb{\cdot, \cdot}_{t, x}$ denotes 
the duality pairing on $\R\times \T^2$
and thus we formally have
\[
B_n(t) = \jb{\xi, \ind_{[0, t]} e_n}_{t, x}
 = \text{``}\int_0^t \int_{\T^2}
\cj{e_n(x)} \xi( dx dt')\text{''}.\]

\noi
As a result, 
we see that $\{ B_n \}_{n \in \Z^2}$ is a family of mutually independent complex-valued\footnote
{In particular, $B_0$ is  a standard real-valued Brownian motion.} 
Brownian motions conditioned so that $B_{-n} = \cj{B_n}$, $n \in \Z^2$. 
By convention, we normalized $B_n$ such that $\text{Var}(B_n(t)) = t$.
Then, for each fixed $x \in \T^2$ and $t \geq 0$,  
we see that  $\Psi_N(t,x)$
 is a mean-zero real-valued Gaussian random variable with variance
\begin{align}
\begin{split}
\s_N(t) &  \stackrel{\text{def}}{=} \E \big[\Psi_N(t,x)^2\big]
 =  \frac{1}{4\pi^2}\sum_{n \in \Z^2}  \chi_N^2(n)
\int_0^t \bigg[\frac{\sin((t - t')\jb{n})}{\jb{n}} \bigg]^2 dt'
\\
& =  \frac{1}{4\pi^2}\sum_{n \in \Z^2} \chi_N^2(n) \bigg\{ \frac{t}{2 \jb{n}^2} - \frac{\sin (2 t \jb{n})}{4\jb{n}^3 }\bigg\}\sim  t \log N
\end{split}
\label{sig}
\end{align}

\noi
\noi
for all $t\in[0,1]$ and $N \gg 1$.
We emphasize that the variance $\s_N(t)$ is time-dependent.
For any $t > 0$, 
we see that  $\s_N(t) \to \infty$ as $N \to \infty$, 
showing that $\{\Psi_N(t)\}_{N \in \N}$
is almost surely unbounded in $W^{0, p}(\T^2)$ for any $1 \leq p \leq \infty$.
See also Lemma \ref{LEM:psi} below.

If we were to proceed with a Picard iteration
to study \eqref{SSGN}, 
the next term we need to study is 
\begin{align}
e^{i\be \Psi_N} = \sum_{k= 0}^\infty\frac{(i\be)^k}{k!}\Psi_N^k.
\label{P1}
\end{align}

\noi
As pointed above,
the power $\Psi_N^k$,  $k \geq 2$,  does not have any nice limiting behavior as $N \to \infty$.
As in \cite{GKO}, 
we now introduce the Wick renormalization:
\begin{align}
:\!\Psi_N^k(t,x)\!:\, \stackrel{\text{def}}{=} H_k\big(\Psi_N(t,x);\s_N(t)\big)
\label{P2}
\end{align}

\noi
to each power $\Psi_N^k$ appearing in \eqref{P1}.
Here,  $H_k$ denotes the $k$th Hermite polynomial, defined through the generating function:
\begin{equation}\label{Hermite}
e^{t x-\frac{\s}2 t^2} = \sum_{k= 0}^\infty\frac{t^k}{k!}H_k(x;\s).
\end{equation}

\noi
From \eqref{P2} and \eqref{Hermite}, 
the renormalized complex exponential is then given by 
 \begin{align}
\begin{split}
 \U_N(t,x)
 &  = \,:\!e^{i\be\Psi_N(t,x)}\!:\,\,  
 \deff
 \sum_{k= 0}^\infty\frac{(i\be)^k}{k!}:\!\Psi_N^k(t, x)\!:\\
 & =  e^{\frac{\be^2}2 \s_N(t)}e^{i\be\Psi_N(t,x)}.
\end{split}
\label{Ups}
 \end{align}

 \noi 
Following \cite{CHS}, we refer to 
$ \U_N$ as the imaginary Gaussian 
 multiplicative chaos.
The following proposition establishes the regularity
and convergence property 
of the imaginary Gaussian multiplicative chaos $\U_N$.

\begin{proposition}\label{PROP:Ups} 
Let $\be\in\R \setminus \{0\}$ and $T>0$
such that $\be^2 T < 8 \pi$.
Then, 
given any finite $p, q\geq 1$ and  any $\al > 0$ satisfying $ \al > \frac{\be^2T}{8\pi} $, the sequence of random variables $\U_N$ is a Cauchy sequence in $L^p(\O;L^q([0,T];W^{-\al,\infty}(\T^2)))$
and hence converges to some limit $\U$
in $L^p(\O;L^q([0,T];W^{-\al,\infty}(\T^2)))$. 

 \end{proposition}

In view of the convergence of the truncated stochastic convolution $\Psi_N$
to $\Psi$ (see Lemma~\ref{LEM:psi}), 
it is natural to write the limit $\U$ as  
\begin{align*}
\U = \, :\!e^{i\be \Psi}\!:.
\end{align*}

\noi
See Remark \ref{REM:app} below
on the uniqueness of the limiting process $\U$.

In the stationary and parabolic settings,
analogous results were established
by Lacoin, Rhodes, and Vargas 
\cite[Theorem 3.1]{LRV}
and Hairer and Shen \cite[Theorem 2.1]{HS}.\footnote{While Theorem 2.1 in \cite{HS}
is stated in terms of space-time regularity, 
it  implies that the conclusion of Proposition \ref{PROP:Ups}  holds 
in the parabolic case, 
provided that  
 $\be^2 < 4 \pi$ and 
 $ \al > \frac{\be^2}{4\pi} $, i.e.~corresponding to 
 the restrictions on $\be$ and $\al $ in 
 Proposition \ref{PROP:Ups} with $T = 2$. }
The main difference between Proposition~\ref{PROP:Ups}
and the previous results in \cite{LRV, HS}
is the dependence of the regularity on the time parameter $T>0$.
In particular, as $T$ increases, the regularity of $\U_N$
gets worse.
On the other hand, 
for fixed $\be \in \R \setminus \{0\}$, by taking $T>0$ small, 
we can take $\{\U_N(t) \}_{N\in \N}$ almost bounded in $L^2(\T^2)$, 
$0 < t \leq T$.

The proof of Proposition \ref{PROP:Ups}
is in the spirit of \cite{LRV,HS}. 
In the case of a polynomial nonlinearity
\cite{GKO, GKO2}, it is enough to estimate  the second moment
and then invoke the Wiener chaos estimate (see, for example, Lemma 2.5 in \cite{GKO2})
to obtain the $p$th moment bound,
since the stochastic objects in \cite{GKO, GKO2}
all belong to Wiener chaoses of finite order. 
The imaginary Gaussian multiplicative
 chaos $\U_N$
 in \eqref{Ups},  however, does \emph{not} belong to any  Wiener chaos of finite order.
 This forces us to estimate all its higher moments by hand. 
 In Section \ref{SEC:2}, we present a proof of Proposition \ref{PROP:Ups}.
 While we closely follow the argument in~\cite{HS}, 
our argument is based on an  elementary calculation. 
See Lemma \ref{LEM:prod}.

\subsection{Main results}
In view of the previous discussion, we are thus led to study 
 the following  renormalized  stochastic sine-Gordon equation:
\begin{align}
\begin{cases}
\dt^2 u_N +(1-\Dl)  u_N   + \g_N\sin (\be u_N)   =  \P_N \xi \\ 
(u_N, \dt u_N) |_{t = 0} = (u_0, u_1),
\end{cases}
\label{RSSGN}
\end{align} 

\noi
where 
 $\g_N$ is defined by 
 \begin{equation}\label{gN}
 \g_N(t,\be) = e^{\frac{\be^2}{2}\s_N(t)}
 \too \infty, 
 \end{equation}

\noi
as $N \to \infty$.
We now state the local well-posedness result.

 \begin{theorem}\label{THM:main}
Let  $\be\in\R\setminus\{0\}$
and $s > 0$.
Given any $(u_0, u_1) \in \H^s(\T^2)
= H^{s}(\T^2) \times H^{s-1}(\T^2)$, 
the Cauchy problem \eqref{RSSGN}
is uniformly locally well-posed
in the following sense;
there exists  
$T_0 =  T_0\big(\|(u_0,u_1)\|_{\H^s}, \be\big)>0$ 
such that given  
any  $0 < T\leq T_0$
and $N \in \N$,  
there exists a set 
$\O_N(T)\subset \O$
such that 
\begin{itemize}
\item[\textup{(i)}]
for any $\o \in \O_N(T)$, 
there exists a unique solution $u_N$ to \eqref{RSSGN} 
in the class
\begin{equation*}
 \Psi_N + X^\s(T) 
\subset C([0, T]; H^{-\eps}(\T^2)), 
\end{equation*}

\noi
for any small $\eps > 0$, 
where $\Psi_N$ is as in \eqref{PsiN}, 
$X^\s(T)$ is  the Strichartz space   defined in~\eqref{X}, 
and  $\s = \min(s, 1-\eps)$, 

\item[\textup{(ii)}] there exists a uniform estimate on the probability of 
the complement of $\O_N(T)$:
\[P(\O_N(T)^c) \too 0,  \]

\noi
 uniformly in $N \in \N$, as $T \to 0$,

\end{itemize}

Furthermore, 
there exist
 an almost surely positive  stopping time $\tau =  \tau\big(\|(u_0,u_1)\|_{\H^s}, \be\big)$ 
 and a stochastic process $u$ in the class
\begin{equation*}
 \Psi + X^\s(T) 
\subset C([0, T]; H^{-\eps}(\T^2))
\end{equation*}

\noi
for any $\eps > 0$
such that, given any small $T >0 $,  
on the event $\{ \tau \geq T\}$, 
the solution $u_N$ to~\eqref{RSSGN}
converges in probability to $u$ in 
$C([0, T]; H^{-\eps}(\T^2))$.
\end{theorem}

Note that,  in Theorem~\ref{THM:main} above, 
we can take 
the parameter $\beta^2 > 0$ arbitrarily large. 
This exhibits sharp contrast with the parabolic case
studied in \cite{HS, CHS}, 
where the parameter was restricted to the subcritical range: $0 < \be^2 < 8 \pi$.

The main point is that Proposition \ref{PROP:Ups}
shows that
  the imaginary 
Gaussian multiplicative chaos $\U_N$ 
is almost   a function (namely, $\al > 0$ small in 
Proposition \ref{PROP:Ups})
by taking $T = T(\be)>0$ sufficiently small.
In particular, we can apply the analysis from \cite{GKO}
on the polynomial nonlinearity
via the Da Prato-Debussche trick
to study \eqref{RSSGN}; namely
write a solution $u_N$ to \eqref{RSSGN} as
\[u_N = \Psi_N + v_N. \]

\noi
In view of \eqref{Ups} and \eqref{gN}, 
the residual term $v_N$ then satisfies
\begin{equation}\label{vN}
\begin{cases}
\dt^2v_N + (1-\Dl)v_N + \Im\big(\U_N e^{i\be v_N}\big)=0\\
(v_N,\dt v_N)\big|_{t=0}=(u_0,u_1).
\end{cases}
\end{equation}

\noi
Then, by taking $T = T(\be) >0$
sufficiently small, thus guaranteeing
the regularity of $\U_N$
via Proposition \ref{PROP:Ups}, 
the standard Strichartz analysis 
as in \cite{GKO} along with the fractional chain rule (Lemma \ref{LEM:FC})
suffices
to conclude Theorem \ref{THM:main}.
As in \cite{GKO}, 
our argument shows that $u$ constructed in Theorem \ref{THM:main}
has a particular structure $u = \Psi + v$, 
where $v \in X^\s(T)$ satisfies 
\begin{equation}
\begin{cases}
\dt^2v + (1-\Dl)v + \Im\big(\U e^{i\be v}\big)=0\\
(v,\dt v)\big|_{t=0}=(u_0,u_1),
\end{cases}
\label{v1}
\end{equation}

\noi
where $\U$ is the limit of $\U_N$ constructed in Proposition \ref{PROP:Ups}.

\smallskip

Next, we consider 
 the  sine-Gordon equation \eqref{SSG} {\it without} renormalization
by studying  its frequency-truncated version:\footnote{Here, we set $\g = 1$ in \eqref{SSG} and \eqref{SSGN} for simplicity.}
 \begin{align}
\begin{cases}
\dt^2 u_N +(1-\Dl)  u_N   + \sin (\be u_N)   =  \P_N \xi \\ 
(u_N, \dt u_N) |_{t = 0} = (u_0, u_1).
\end{cases}
\label{SSG2}
\end{align} 

 \noi
 In studying the limit as $N \to \infty$,
 we expect the solution $u_N$ to become singular.
 As a result, in view of faster and faster ``oscillation'', 
 we expect 
$\sin (\be u_N)$  to tend to 0 as a space-time distribution. 
This is the reason why we needed to insert 
a diverging multiplicative constant $\g_N$
in the renormalized model \eqref{RSSGN}.

\begin{remark}\label{REM:app}\rm 
In Proposition \ref{PROP:Ups}
and Theorem \ref{THM:main} above, 
we used a smooth frequency projector 
$\P_N$ with the multiplier $\chi_N$ in \eqref{chi}.
As in the parabolic case, 
it is possible to show that 
 the limiting process 
$\U$ of $\U_N$ in Proposition~\ref{PROP:Ups}
and 
the limit $u$ of $u_N$ in Theorem~\ref{THM:main}
are independent of the choice of the smooth cutoff function $\chi$.
See \cite{OOTz} for such an argument
in the wave case (with a polynomial nonlinearity).
Moreover, 
we may also proceed by smoothing via a mollification
and obtain analogous results.
In this case, the limiting processes agree with 
those constructed in 
 Proposition~\ref{PROP:Ups}
and Theorem~\ref{THM:main}.

\end{remark}

In the following, 
we study the limiting behavior of $u_N$, 
solving the {\it unrenormalized} model~\eqref{SSG2}
with regularized noises,
and establish a triviality result.
The heuristics above indicates that the nonlinearity 
$\sin (\be u_N) $ tends to 0 in some suitable sense,
indicating that $u_N$ converges to 
a solution to the linear stochastic wave equation.
The next proposition shows that this is indeed the case.

\begin{proposition}\label{PROP:triviality}
Let  $\be\in\R\setminus\{0\}$
and 
fix  $(u_0, u_1) \in \H^s(\T^2)$ for some $s > 0$.
Given any small $T>0$, 
the solution $u_N$ to \eqref{SSG2}
converges in probability 
to 
the solution $u$, satisfying  the following linear stochastic wave equation:
\begin{equation}\label{LSW}
\begin{cases}
\dt^2 u + (1-\Dl) u = \xi\\
(u,\dt u)|_{t=0}=(u_0,u_1)
\end{cases}
\end{equation}

\noi
in the class $C([0,T];H^{-\eps}(\T^2))$, 
$\eps > 0$, as $N \to \infty$.

\end{proposition}

We point out that the nature of ``triviality''
in  Proposition~\ref{PROP:triviality} is slightly different 
from that in the work  \cite{HRW,OPT1,OOR}
with the cubic nonlinearity, where solutions $u_N$
to the unrenormalized equation with regularized noises
tend to $0$ as we remove the regularization.
This  difference comes from the different renormalization procedures;
 our renormalization for the sine-Gordon model
appears as a multiplication by the renormalization constant.
On the other hand,  in the case of the cubic nonlinearity, 
the renormalization is implemented by insertion of a linear term $\g_N u_N$
with a suitable divergent renormalization constant $\g_N$,
thus modifying the linear part of the equation.
In particular, in the case of the wave equation  
the modified linear propagator introduces faster and faster oscillations
and thus the contribution from the deterministic initial data $(u_0, u_1)$
also vanishes.
See \cite{OPT1,OOR}.

Following
the previous work \cite{HRW,OPT1,OOR}, the main idea for proving Proposition \ref{PROP:triviality}
is to artificially insert the renormalization constant $\g_N$
in the unrenormalized equation \eqref{SSG2}.
By using the decomposition $u_N = \Psi_N + v_N$
as before, 
we  have
\begin{equation*}
\begin{cases}
\dt^2v_N + (1-\Dl)v_N + \g_N^{-1} \Im\big(\U_N e^{i\be v_N}\big)=0\\
(v_N,\dt v_N)\big|_{t=0}=(u_0,u_1).
\end{cases}
\end{equation*}

\noi
Proposition \ref{PROP:triviality}
then follows from 
essentially repeating
the proof of Theorem \ref{THM:main}
along with the asymptotic behavior \eqref{gN}
of the renormalization constant.

\begin{remark}\label{REM:Gibbs}\rm

In the case of a  polynomial nonlinearity, 
the local-in-time analysis in \cite{GKO} essentially applies 
to 
the following hyperbolic $\Phi^{2k+2}_2$-model:\footnote{This is the so-called
``canonical'' stochastic quantization equation.  See \cite{RSS}.}
\begin{equation}\label{SdNLW}
\dt^2u + \dt u+ (1-\Dl)u + u^{2k+1} =\sqrt{2}\xi
\end{equation}

\noi
with random initial data distributed by the Gibbs measure
(= $\Phi^{2k+2}_2$-measure).
Then, Bourgain's invariant measure argument \cite{BO94, BO96}
allows us to establish
almost sure global well-posedness of 
(a renormalized version of) \eqref{SdNLW}
and invariance of the Gibbs measure.
See~\cite{GKOT, ORTz}.

By adding the damping term $\dt u$ to \eqref{SSG}, 
we obtain the following 
``canonical'' stochastic sine-Gordon model:
\begin{equation}\label{SdSG}
\dt^2u+ \dt u + (1-\Dl)u + \g\sin(\be u) = \sqrt{2}\xi.
\end{equation}

\noi
This equation is a hyperbolic counterpart
of the stochastic quantization equation
for the quantum sine-Gordon model;
the equation \eqref{SdSG}  formally preserves the Gibbs measure $\mu$
given by\footnote{The Gibbs measure $\mu$ obviously requires a proper renormalization on $\T^2$.}
\[ d\mu = Z^{-1} \exp\bigg( - \frac12 \int_{\T^2} |\jb{\nb} u|^2 dx-
  \frac12 \int_{\T^2} (\dt  u)^2 dx + \frac{\g}{\be}\int_{\T^2} \cos (\be u) dx\bigg) du \otimes d (\dt u).
 \]

\noi
Hence, it is of great importance to study the well-posedness issue of \eqref{SdSG}.
For this model, 
by applying an analogous  renormalization
and repeating the analysis in Section \ref{SEC:2} below, 
we see that  the renormalization constant is time-independent
(as in the parabolic model~\cite{HS, CHS}).
In particular, the regularity depicted in Proposition \ref{PROP:Ups}
for the relevant imaginary Gaussian chaos $\U$
will be time-independent in this case,
implying that there exists an infinite number of threshold
values for $\be^2$ as in the parabolic case.
By drawing an analogy to 
the $\Phi^3_d$-model (see \cite{HS}), 
we see that  the $\Phi^3_3$-model corresponds to $\be^2 = 2\pi$.
In~a recent preprint~\cite{ORSW2},
we establish
almost sure global well-posedness of (a renormalized version of) \eqref{SdSG}
and invariance of the Gibbs measure
for  $0 < \be^2 < 2\pi $.

In view of a recent work~\cite{GKO2},
as for the stochastic nonlinear wave equation 
on $\T^3$
with a quadratic nonlinearity, 
we expect  that it would require a significant effort
(such as establishing multilinear smoothing and introducing appropriate paracontrolled operators)
 to reach 
 the first non-trivial threshold $\be^2 = 2\pi$ for
 the hyperbolic sine-Gordon model~\eqref{SdSG}.\footnote
{We point that the hyperbolic $\Phi^3_3$-model treated 
in \cite{GKO2} is much harder than the parabolic $\Phi^3_3$-model
studied in \cite{EJS}, where a standard application of the Da Prato-Debussche trick suffices.
Due to the non-polynomial nature of the problem, 
we expect the hyperbolic sine-Gordon model \eqref{SdSG} with $\be^2 = 2\pi$ to be even harder than 
the hyperbolic $\Phi^3_3$-model.}

\end{remark}

\begin{remark}\rm

 Albeverio, Haba, and Russo \cite{AHR} considered the following stochastic nonlinear wave equation:
\begin{equation}\label{gSNLW}
\dt^2u +(1-\Dl) u + f(u)=\xi
\end{equation}
with zero initial data, 
where $f$ is assumed to be smooth and bounded. 
Working within the framework of Colombeau generalized functions,
 they showed that 
the solutions $u_N$ to~\eqref{gSNLW} with regularized noises
tend to the solution $u$  to the stochastic linear wave equation~\eqref{LSW}.
It is, however, not clear what the meaning of the solutions constructed 
in \cite{AHR} and how they relate to the construction in this paper.
Note that in our triviality result (Proposition~\ref{PROP:triviality}), 
we establish the convergence in 
the natural space
$C([0, T]; H^{-\eps}(\T^2))$.
See also the comments in \cite{HS, CHS} on the work \cite{AHR2}
in the parabolic setting.

\end{remark}

\smallskip

\noi
{\bf Notations:} Before proceeding further, we introduce some notations here.
We set
\begin{align}
e_n(x) \stackrel{\textup{def}}{=} \frac1{2\pi}e^{in\cdot x}, 
\qquad  n\in\Z^2, 
\label{exp}
\end{align}
 
\noi 
for the orthonormal Fourier basis in $L^2(\T^2)$. 
Given $s \in \R$, 
we define the Bessel potential $\jb{\nb}^s$ of order $-s$ as a Fourier multiplier operator
given by 
\[ \jb{\nb}^s f = \F^{-1} (\jb{n}^s \ft f (n)), \]

\noi
where $\F^{-1}$ denotes the inverse Fourier transform  and  $\jb{\,\cdot\,} = (1+|\cdot|^2)^\frac{1}{2}$. 
Then, 
we  define the  Sobolev space  $H^s (\T^2)$ by  the norm:
\[
\|f\|_{H^s(\T^2)} = \|\jb{\nb}^s f \|_{L^2} = \|\jb{n}^s\ft{f}(n)\|_{\l^2(\Z^2)}.
\]

\noi
We also set 
\begin{equation*}
\H^{s}(\T^2)  \stackrel{\textup{def}}{=} H^{s}(\T^2) \times H^{s-1}(\T^2).
\end{equation*}

\noi
We use short-hand notations such as
 $C_TH^s_x = C([0, T]; H^s(\T^2))$
 and $L^p_\o = L^p(\O)$.

For $A, B > 0$, we use $A\lesssim B$ to mean that
there exists $C>0$ such that $A \leq CB$.
By $A\sim B$, we mean that $A\lesssim B$ and $B \lesssim A$.
We also use  a subscript to denote dependence
on an external parameter; for example,
 $A\les_{\al} B$ means $A\le C(\al) B$,
  where the constant $C(\al) > 0$ depends on a parameter $\al$. 
Given two functions $f$ and $g$ on $\T^2$, 
we write  
\begin{align}
f\approx g
\label{approx}
\end{align}

\noi
 if there exist some constants $c_1,c_2\in\R$ such that $f(x)+c_1 \leq g(x) \leq f(x)+c_2$ for any $x\in\T^2\backslash \{0\}
 \cong [-\pi, \pi)^2 \setminus\{0\}$.
Given $A, B \geq 0$, we also set
 $A\vee B = \max(A, B)$
 and
 $A\wedge B = \min(A, B)$.

 \section{On the imaginary Gaussian multiplicative chaos}
 \label{SEC:2}

In this section, we study the regularity and convergence properties of 
the imaginary complex Gaussian chaos $\U_N
= \, :\!e^{i\be\Psi_N}\!:$ defined in \eqref{Ups} 
and present a proof of 
Proposition~\ref{PROP:Ups}.
As pointed out in the introduction, 
the main difficulty arises in  that the processes $\U_N$  do not belong 
to any Wiener chaos of finite order.
This forces us to 
 estimate all the higher moments by hand.
As in \cite{HS}, 
the main ingredient 
for proving Proposition~\ref{PROP:Ups} is 
the bound~\eqref{Ebd} 
 in  Lemma~\ref{LEM:prod}, 
 exhibiting a certain  cancellation property.
This  
 allows us to simplify the expressions of the moments.

 \subsection{Preliminary  results}

We first recall the Poisson summation formula
 (see \cite[Theorem 3.2.8]{Gra1}).
In our context, it reads as follows.
Let $f \in L^1(\R^d)$
such that (i) $|f(x)| \leq \jb{x}^{-d - \dl}$ for some $\dl > 0$
and any $x \in \R^d$
and (ii) $\sum_{n \in \Z^d} |\ft f(n)| < \infty$, 
where $\ft f $ denotes the Fourier transform of $f$ on $\R^d$
defined by 
\[ \ft f(\xi) = \frac 1{(2\pi)^\frac{d}{2}} \int_{\R^d} f(x) e^{-i\xi \cdot x} d x.\]

\noi
Then, we have 
\begin{equation}\label{Poisson}
\sum_{n\in\Z^d}\ft f(n)e_n(x) = \sum_{m\in\Z^d}f(x+2\pi m)
\end{equation}

\noi
for any $x \in \R^d$.

We also recall the following calculus  lemma from \cite[Lemma 3.2]{HRW}.
\begin{lemma}\label{LEM:HRW}
There exists  $C>0$ such that
\[\bigg|\sum_{\substack{n \in \Z^2\\|n|\leq R}}\frac{1}{a+|n|^2}-\pi\log\bigg(1+\frac{R^2}a\bigg)\bigg|\leq \frac{C}{\sqrt{a}}\min\bigg(1,\frac{R}{\sqrt{a}}\bigg)\]

\noi
for any  $a,R\geq 1$. 
\end{lemma}

In  \cite{OTh2, GKO, GKO2}, the analysis of stochastic objects
 was carried out on the Fourier side.
 It turns out that for our problem, 
 it is more convenient to work on the physical side.
 For this reason, we recall several 
 facts from harmonic analysis. 
 For $d\in\N$ and $\al>0$, 
 we denote  by $\jb{\nabla}^{-\al} = (1-\Dl)^{-\frac{\al}2}$
 the Bessel potential of order $\al$. 
 This operator on $\T^d$ is given by  a convolution with the 
 following distribution:
 \begin{align}\label{Bessel}
 J_{\al}(x)  \deff \lim_{N \to \infty}\frac{1}{2\pi} \sum_{n \in\Z^d} \frac{\chi_N(n)}{\jb{n}^{\al}} e_n(x),
 \end{align}
 
 \noi
 where $\chi_N$ is the smooth cutoff function 
 onto the frequencies $\{|n|\les N\}$ defined in \eqref{chi}.
 It is then known that for $0<\al<d$, 
 the  distribution $J_\al$ behaves like $|x|^{-d + \al}$ modulo an additive smooth function on $\T^d$.
 More precisely, we have the following lemma.
 
 \begin{lemma}
 \label{LEM:Bessel}
 Let $0 < \alpha < d$. Then,  the function $J_\al $ in \eqref{Bessel}
 is smooth away from the origin and integrable on $\T^d$.
 Furthermore, there exist a constant $c_{\al,d}$ and a smooth function $R$ on $\T^d$ 
 such that 
\[ J_\al(x) = c_{\al,d} |x|^{\al - d} +R(x)\]

\noi
for all $x\in\T^d\setminus\{0\}
 \cong [-\pi, \pi)^d \setminus\{0\}$.
 \end{lemma}

 \begin{proof}
 
 Let $K_{\al}$ be the convolution kernel 
 for the Bessel potential on $\R^d$.
 Then, 
it follows from \cite[Proposition 1.2.5]{Gra2}
 that $K_\al$ is a smooth function on $\R^d\setminus\{0\}$
 and decays exponentially as $|x|\to \infty$.
 Moreover, 
 the following  asymptotic behavior holds:
 \[K_\al(x) = c_{\al,d} |x|^{\al - d} + o(1)\]
 
 \noi
  as $x \to 0$.
 See (4,2) in \cite{AS}.
Then,  this lemma
follows from 
applying the Poisson summation formula \eqref{Poisson}
with a frequency truncation $\P_N$ and
then taking $N \to \infty$
as in \cite[Example 3.2.10]{Gra1}.
See also \cite{BO}.
 \end{proof}

Next, we study the Green's function for $1- \Dl$.
Recall from \cite[Proposition 1.2.5]{Gra2}
that 
the Green's function $G_{\R^2}$ for $1 - \Dl$
on $\R^2$
is a smooth function on $\R^2 \setminus \{0\}$
and  decays exponentially as $|x| \to \infty$.
Furthermore, it satisfies
\begin{align}
G_{\R^2}(x) = -\frac1{2\pi} \log|x| + o(1) 
\label{G1}
\end{align}

\noi
as $x \to 0$.
See  (4,2) in \cite{AS}.
Now, 
let $G$ be the Green's function
for $1-\Dl$ on $\T^2$.
In view of our normalization \eqref{exp}, we then have
\begin{align}
G \deff (1-\Delta)^{-1} \dl_0 = \frac{1}{2\pi}\sum_{n\in\Z^2}\frac 1{\jb{n}^{2}}e_n,
\label{G1a}
\end{align}

\noi
where the sum is to be interpreted in the limiting sense as in \eqref{Bessel}.
Then, by applying the Poisson summation formula \eqref{Poisson}
(with a frequency truncation $\P_N$ and taking $N \to \infty$)
together with the asymptotics \eqref{G1}, 
we obtain
\begin{align}
G(x) = -\frac1{2\pi} \log|x| +  R(x), \qquad x \in \T^2 \setminus \{0\}, 
\label{G2}
\end{align}

\noi
for some smooth function $ R$.

In the next lemma, we establish an analogous 
behavior for the frequency-truncated Green's function $\P_N^2 G$ given by 
\begin{align}
\P_N^2G(x) = \frac{1}{2\pi}\sum_{n\in\Z^2}\frac{\chi_N^2(n)}{\jb{n}^2}e_n(x).
\label{G3}
\end{align}

\noi
See also Lemma 3.7 in  \cite{HS}.

\begin{lemma}\label{LEM:Green}
For any $N\in\N$ and $x\in\T^2\setminus \{0\}$,
we have 
\[\P_N^2G(x) \approx -\frac{1}{2\pi} \log\big(|x|+N^{-1}\big),\]

\noi
where the notation $\approx$ is as in \eqref{approx}.
\end{lemma}

\begin{proof}
Fix $x\in\T^2 \setminus\{0\}\cong [-\pi,\pi)^2\setminus\{0\}$. 
We separately consider   the cases $|x|\lesssim N^{-1}$
and $|x|\ges  N^{-1}$. 

\smallskip
\noi
$\bullet$
{\bf Case 1:} 
We first consider  the case $|x|\lesssim N^{-1}$. 
By  the mean value theorem, we have 
\begin{align}
\big|\P_N^2G(x) - \P_N^2G(0)\big| 
\sim \bigg|\sum_{n\in \Z^2}\frac{\chi_N^2(n)}{\jb{n}^{2}}(e_n(x)-e_n(0))\bigg| 
\les \sum_{|n|\les N} \frac{|x|}{\jb{n}} 
\les N |x| \les 1.
\label{A1}
\end{align}

\noi
On the other hand, 
from Lemma~\ref{LEM:HRW},
we have
\begin{align}
\P_N^2G(0) = \frac{1}{4\pi^2}\sum_{n\in\Z^2}\frac{\chi_N^2(n)}{\jb{n}^2}=\frac{1}{2\pi}\log N + O(1)
\label{A2}
\end{align}

\noi
as
$N\to \infty$.
Hence, from \eqref{A1} and \eqref{A2}, we obtain
\begin{align*}
\P_N^2G (x) \approx -\frac1{2\pi}\log N^{-1}
\approx -\frac1{2\pi}\log \big(|x| + N^{-1}\big)
\end{align*}
 in this case.

\smallskip
\noi
$\bullet$
{\bf Case 2:} 
Next, we consider the case $|x|\ges N^{-1}$.
Let  $\rho\in C^{\infty}(\R^2)$ with a compact support in $\T^2\cong [-\pi,\pi)^2$ with 
$\int\rho(y)dy = 1$. 
Given $M \geq N$, 
let $\rho_{M}$ denote a mollification kernel of scale $M^{-1}$
given by $\rho_{M}(x)=M^2\rho(Mx)$. 
Then, 
since $G$ and $\P_N^2G$ are smooth  away from the origin,  we have
\begin{align}
\P_N^2G(x)-G(x)=\lim_{M\rightarrow\infty}\rho_{M}*\big[(\P_N^2-1)G\big](x).
\label{A3}
\end{align}

\noi
Then, by the Poisson summation formula \eqref{Poisson}, we write
\begin{align}
\begin{split}
\rho_{M}*\big[(\P_N^2-1)G\big](x) 
& = \frac{1}{2\pi}\sum_{n\in\Z^2}\ft \rho_{M}(n) \frac{\chi_N^2(n)-1}{\jb{n}^{2}}e_n(x)\\
& =\frac 1{2\pi} \sum_{m\in\Z^2}f_{N,M}(x+2\pi m),
\end{split}
\label{A4}
\end{align}

\noi
where  $f_{N, M} \in \S(\R^2)$ is defined by its Fourier transform:
\[\ft f_{N,M}(\xi) = \ft \rho_{M}(\xi)(\chi_N^2(\xi)-1)\jb{\xi}^{-2}\in\S(\R^2).\]

By applying integrations by parts with the properties of $\chi_N$ and $\rho_{M}$, 
in particular, the fact that the integrand is essentially\footnote
{Namely, thanks to the fast decay of $\ft \rho_{M}(\xi)$, 
the  contribution to the integral in the second line of  \eqref{A5} from $\{ |\xi| \gg M\}$
can be easily bounded by $M^{-2k}$ for any $k \in \N$.}
 supported on $\{N \les |\xi|\les M\}$, 
we obtain
\begin{align}
\begin{split}
\big|f_{N,M}(x+2\pi m)\big| 
&  = \frac{1}{2\pi} \bigg|\int_{\R^2}\ft f_{N,M}(\xi)e^{i\xi \cdot (x+2\pi m)}d\xi\bigg|\\
&\sim |x+2\pi m|^{-2k}
\bigg|\int_{\R^2}\Dl_\xi^k\big(\ft\rho_{M}(\xi)(\chi_N^2(\xi)-1)
\jb{\xi}^{-2}\big)e^{ i\xi\cdot (x+2\pi m)}d\xi\bigg|\\
& \les N^{-2k}|x+2\pi m|^{-2k}
\end{split}
\label{A5}
\end{align}

\noi
for any $k\in\N$, uniformly in  $M\geq N \geq 1$.
When a derivative hits the second factor $\chi_N^2(\xi)-1$, 
we used the fact that the support in that case is essentially contained in 
$\{|\xi|\sim N\}$.
When no derivative hits the second factor $\chi_N^2(\xi)-1$ 
or the third factor $\jb{\xi}^{-2}$, 
namely, when all the derivatives hit the first factor
$\ft \rho_{M}(\xi)=\ft\rho(M^{-1}\xi)$, 
we used the following bound:
\begin{align*}
M^{-2k} \log \frac{M}{N}
= 
N^{-2k}  \bigg(\frac{N}{M}\bigg)^{2k} \log \frac{M}{N}
 \les N^{-2k}.
\end{align*}

By choosing $k > 1$, 
we  perform the summation over $m \in \Z^2$ 
 and take the limit $M\rightarrow \infty$ in~\eqref{A4}.
From \eqref{A3}, \eqref{A4}, and \eqref{A5}, 
we obtain
\begin{align}
\begin{split}
|\P_N^2G(x)-G(x)|
& \les  N^{-2k}|x|^{-2k} + \sum_{m\neq 0}N^{-2k}|m|^{-2k}  \les N^{-2k}(|x|^{-2k}+1)\\
& \les (N|x|)^{-2k}
\les 1,
\end{split}
\label{A6}
\end{align}

\noi
under our assumption  $|x|\gtrsim N^{-1}$, 
where we used that $|x+2\pi m|\gtrsim |m|$ for $x\in\T^2\cong  [-\pi,\pi)^2$ and $m\neq 0$.
Finally, from \eqref{G2} and \eqref{A6}, we obtain
\[\P_N^2G (x)\approx
G(x)  \approx  -\frac{1}{2\pi}\log |x|
 \approx -\frac{1}{2\pi}\log\big(|x|+N^{-1}\big).\]

\noi
This completes the proof of Lemma \ref{LEM:Green}.
\end{proof}

\begin{remark}\label{REM:G1}\rm
Let $N_2 \geq N_1 \geq 1$.
Then, by slightly modifying the proof of Lemma \ref{LEM:Green}
(namely, replacing $\chi_N^2$ by $\chi_{N_1} \chi_{N_2}$), 
we have
\begin{align}
\P_{N_1}\P_{N_2} G(x) \approx -\frac{1}{2\pi} \log\big(|x|+N_1^{-1}\big).
\label{A6b}
\end{align}

\noi
Similarly, we have
\begin{align}
|\P_{N_j}^2 G(x) - \P_{N_1}\P_{N_2} G(x)| 
\les  \big(1 \vee - \log\big(|x|+N_2^{-1}\big)\big) \wedge \big(N_1^{-1}|x|^{-1} \big)
\label{A6a}
\end{align}

\noi
for $j = 1, 2$.
We point out that the second bound by $N_1^{-1}|x|^{-1}$  
in \eqref{A6a}
follows from the computation up to \eqref{A6} in Case~2 
of the proof of Lemma \ref{LEM:Green}
(without dividing the argument, depending on the size of $|x|$).

\end{remark}

The next lemma plays a crucial role in proving Proposition \ref{PROP:Ups},
allowing us to reduce the  product
of $p(2p-1)$ factors to that of  $p$ factors.

  \begin{lemma}
  \label{LEM:prod}
 Let $\lambda>0$ and $p \in\N$.
 Given $j \in \{1, \dots, 2p\}$, 
 we set 
 $\eps_j = 1$ if $j$ is even and $\eps_j = -1$ 
 if $j$ is odd. 
  Let $S_p$ be the set of all permutations of the set $\{1,2,\dots, p\}$. Then, the following estimate holds:
  \begin{equation}
  \label{Ebd}
\prod_{1\leq j< k \leq 2p} \big(|y_j - y_k| + N^{-1}\big)^{ \eps_j\eps_k\lambda}
  \les \max_{\tau \in S_p} \prod_{1\leq j \leq p} \big(|y_{2j} - y_{2\tau(j)-1}| + N^{-1}\big)^{ -\lambda}
   \end{equation}

\noi
for any set  $\{y_j\}_{j=1,\dots,2p}$ of $2p$ points in $\T^2$
and any $N \in \N$.
  \end{lemma}

In \cite{HS}, Hairer and Shen established an analogous result
for a more general class of ``potential'' function $\J_N(x-y)$
than $(|x-y|+N^{-1})^{\ld}$ considered above.
They exploited
 sophisticated multi-scale analysis 
 and charge cancellation. 
In particular, 
Lemma~\ref{LEM:prod} follows from 
Proposition~3.5 in~\cite{HS}.
Nonetheless, we decided to include a proof 
since 
the relevant computation 
is elementary and concise 
in our concrete setting.
In the proof below,  the charge cancellation 
due to the ``dipole'' formation appears in  \eqref{Era2} and \eqref{Era3}. 
We also mention  the work \cite{Fro}
where an exact identity (see (3.13) in \cite{Fro}) was used to study a similar problem.

  \begin{proof}
  We prove \eqref{Ebd} by induction on $p$.
When $p = 1$, \eqref{Ebd} holds with an equality.
  Now, let us assume that \eqref{Ebd} holds for some $p\in\N$
  and fix  a set $\{y_j\}_{j=1,...,2(p+1)}$ of $2(p+1)$ points in $\T^2$.
By  defining the sets $A_p^+ =  \{j=1,...,2p,~j\text{ even}\}$ and $A_p^- = \{k=1,...,2p, ~k\text{ odd}\}$, 
we set 
\begin{align*}
A_p^{\s_1 \s_2} = \big\{ (j, k) \in \{1, \dots, 2p\}:  j < k, \
j \in A_p^{\s_1}, \ k \in A_p^{\s_2}\big\} 
\end{align*}

\noi
for $\s_1, \s_2\in \{ +, -\}$.
With a slight abuse of notation, we identify $+$ and $-$ with $+1$ and $-1$ in the following.

We denote by  $\Pi_{p}$ the left-hand side of \eqref{Ebd}.  
Then, we have 
\begin{align*}
\Pi_{p+1} =   \frac{\displaystyle
  \prod_{\substack{\s_1, \s_2 \in \{+, -\}\\
\s_1\s_2 = 1}}\prod_{(j,k)\in A_{p+1}^{\s_1\s_2}} \big(|y_j - y_k| + N^{-1}\big)^{ \lambda}
  }{\displaystyle
  \prod_{\substack{\s_1, \s_2 \in \{+, -\}\\
\s_1\s_2 = -1}}\prod_{(j,k) \in A_{p+1}^{\s_1\s_2}}
  \big(|y_j - y_k| + N^{-1}\big)^{ \lambda}} .
\end{align*}

\noi
Namely, the numerator contains all the factors
with the same parity for $j$ and $k$, 
while the denominator contains all the factors
with the opposite parities for $j$ and $k$.
 Now,  
 choose $j_0 \in A_{p+1}^+ $ and $k_0 \in A_{p+1}^- $ such that
\begin{align}
\label{min}
|y_{j_0} - y_{k_0}| = \min \big\{ |y_j - y_k|; j \in A_{p+1}^+, k \in A_{p+1}^-\big\}.
\end{align}
  
  \noi
It follows from its definition given by the left-hand side of \eqref{Ebd}
(with $p$ replaced by $p+1$)
that 
 $\Pi_{p+1}$ is invariant under permutations
 of $\{1, \dots, 2p+2\}$
that do not mix even and odd integers
(i.e.~not mixing the $+$ and $-$ charges).
Thus, without loss of generality, 
 we assume that $(j_0,k_0) = (2p+1, 2p+2)$.

By  the inductive hypothesis we have
\[\Pi_p \les  \max_{\tau \in S_p} 
\prod_{1\le j\le p} \big(|y_{2j} - y_{2\tau(j)-1}| + N^{-1}\big)^{- \lambda}.
\]   

\noi
In view of \eqref{Ebd},  it suffices to prove the following bound:
\begin{align}
\label{Era}
\frac{\Pi_{p+1}}{\Pi_p} \les \big(|y_{2p+1} - y_{2p+2}| + N^{-1}\big)^{- \lambda}.
\end{align}

\noi
uniformly in $N$.
Note that the left-hand side of \eqref{Era}
contains only the factors involving $y_{2p+1}$ or $y_{2p+2}$.
Hence, we have
  \begin{align}
  \begin{split}
    \frac{\Pi_{p+1}}{\Pi_p}  
    \big( & |y_{2p+1} - y_{2p+2}| + N^{-1}\big)^{\lambda}\\
&  = \prod_{1\le j \le p} 
\bigg( \frac{|y_{2j-1} - y_{2p+1}| + N^{-1}}{|y_{2j-1} - y_{2p+2}| + N^{-1}}\bigg)^{\lambda}
 \bigg( \frac{|y_{2j} - y_{2p+2}| + N^{-1}}{|y_{2j} - y_{2p+1}| + N^{-1}}\bigg)^{\lambda} 
  \end{split}
\label{A7}
\end{align} 
  
\noi  
On the other hand, 
by the triangle inequality and \eqref{min}, we have  
\begin{align}
\label{Era2}
\max_{1\le j \le p}  \frac{|y_{2j-1} - y_{2p+1}| + N^{-1}}{|y_{2j-1} - y_{2p+2}| + N^{-1}} 
\leq 1 + \max_{1\le j \le p}  \frac{ |y_{2p+1} - y_{2p+2}|}{|y_{2j-1} - y_{2p+2}| + N^{-1}} \les 1, 
\end{align}

\noi
uniformly in $N$.
Similarly, we have
  \begin{align}
  \label{Era3}
    \max_{1 \le j \le p}  \frac{|y_{2j} - y_{2p+2}| + N^{-1}}{|y_{2j} - y_{2p+1}| + N^{-1}} \les 1,
  \end{align}
uniformly in $N$. 
Therefore, the desired estimate \eqref{Era} follows
from \eqref{A7}, 
 \eqref{Era2}, and~\eqref{Era3}.
  \end{proof}

\subsection{Estimates on the stochastic objects}

In this subsection, we present a proof of Proposition \ref{PROP:Ups}
on the imaginary complex Gaussian chaos $\U_N
= \, :\!e^{i\be\Psi_N}\!:$ defined in \eqref{Ups}. 

We first recall 
 the following lemma from \cite[Proposition 2.1]{GKO} on the regularity of the 
 truncated stochastic convolution $\Psi_N$.
See also \cite[Lemma 3.1]{GKO2}.

\begin{lemma}\label{LEM:psi}
Given any  $T,\eps>0$ and finite $p \geq 1$, 
 $\{\Psi_N\}_{N\in \N}$ is a Cauchy sequence in $L^p(\O;C([0,T];W^{-\eps,\infty}(\T^2)))$,
 converging to some limit $\Psi$ in $L^p(\O;C([0,T];W^{-\eps,\infty}(\T^2)))$.
Moreover,  $\Psi_N$  converges almost surely to the same  limit $\Psi\in C([0,T];W^{-\eps,\infty}(\T^2))$.
\end{lemma}

It is easy to show that  the claim in Lemma \ref{LEM:psi} fails
when $\eps = 0$.
Thus, as $N\rightarrow \infty$, $\Psi_N$ becomes singular, 
necessitating a proper  renormalization procedure for $\U_N$ defined in~\eqref{Ups}.

As mentioned above, 
we will work on the physical space (rather than on the Fourier side as 
in  \cite{OTh2, GKO, GKO2}).
For this purpose, we first study the 
 covariance function for $\Psi_N$. Let 
\begin{equation}\label{def-covar}
\G_N(t,x-y) \deff \E\big[\Psi_N(t,x) \Psi_N(t,y)\big].
\end{equation}

\begin{lemma}
\label{LEM:covar}
Given $N \in \N$, let $\G_N$ be as in \eqref{def-covar}.
Then, we have 
\begin{equation}
\label{covar}
\G_N(t,x-y) =\frac1{2\pi}\sum_{n\in\Z^2} \chi_N^2(n)  \bigg\{ \frac{t}{2 \jb{n}^2} - \frac{\sin (2 t \jb{n})}{4\jb{n}^3 }\bigg\}e_n (x-y)  .
\end{equation}

\noi
In particular, we have 
\begin{align}
\label{bd-GN}
\G_N(t,x-y) \approx - \frac{t}{4\pi} \log \big(|x-y| + N^{-1}\big)
\end{align}

\noi
 for any $t \in \R$.

\end{lemma}

\begin{proof}
The identity \eqref{covar}
follows from a straightforward computation,
using \eqref{PsiN} and~\eqref{W1}
with \eqref{exp}.
See~(1.8) and~(2.8) in~\cite{GKO} for analogous computations.

From \eqref{G3} and \eqref{covar}, we have
\[ \G_N (t, x) \approx \frac{t}{2} \P_N^2 G(x).\]

\noi
Then, \eqref{bd-GN} follows from 
Lemma~\ref{LEM:Green}.
\end{proof}

By setting 
\begin{equation*}
\G(t,x-y) \deff \E\big[\Psi(t,x) \Psi(t,y)\big], 
\end{equation*}

\noi
we then formally have
\begin{equation}
\label{G4}
\G(t,x-y) =\frac1{2\pi}\sum_{n\in\Z^2}   \bigg\{ \frac{t}{2 \jb{n}^2} - \frac{\sin (2 t \jb{n})}{4\jb{n}^3 }\bigg\}e_n (x-y)  ,
\end{equation}

\noi
where the sum is to be interpreted in the limiting sense as in \eqref{Bessel}.
By comparing \eqref{G1a} and \eqref{G4}, 
we have $\G(t, x) \approx \frac t2 G(x)$.
Similarly, we have
\begin{align}
\P_{N_1} \P_{N_2} \G (t, x)
 \approx \frac{t}{2} \P_{N_1}\P_{N_2} G(x).
\label{G5}
\end{align}

\noi
In particular, from \eqref{A6b} in Remark \ref{REM:G1}, 
we obtain
\begin{align}
\P_{N_1}\P_{N_2} \G(t, x) \approx -\frac{t}{4\pi} \log\big(|x|+N_1^{-1}\big)
\label{G6}
\end{align}

\noi
for $N_2 \geq N_1 \geq 1$.

\smallskip

We are now ready to present a proof of  Proposition~\ref{PROP:Ups}.

\begin{proof}[Proof of Proposition~\ref{PROP:Ups}.]
Fix $\be \in\R$ and $T > 0$ such that $\be^2 T < 8 \pi$.
Also let $p \in \N$, finite $q \geq 1$, 
and $\al >  \frac{\be^2 T}{8\pi} $.
Without loss of generality, we assume $\al < 2$ in the following.

Fix small  $\dl > 0$.
Then, by Sobolev's inequality (in $x$) followed by 
Minkowski's integral inequality, we have
\begin{align}
\|\U_N\|_{L^{2p}_\o L^q_TW^{-\al,\infty}_x} \les \|\jb{\nabla}^{\dl-\al}\U_N\|_{L^{2p}_\o L^q_TL^{r_{\dl}}_x}\les \Big\|\|\jb{\nabla}^{\dl-\al}\U_N(t,x)\|_{L^{2p}_\o}\Big\|_{L^q_TL^{r_{\dl}}_x}
\label{E0a}
\end{align}

\noi
for some large but finite $r_{\dl}$,
provided that 
 $2p\geq \max(q,r_{\dl})$.

Fix $t\in[0,T]$ and $x\in\T^2$. 
Recalling $\jb{\nb}^{\dl - \al} f = J_{\al-\dl} * f$, 
where $J_\al$ is as in \eqref{Bessel}, 
it follows from \eqref{Ups} that
 \begin{align}
 \begin{split}
\E & \Big[\big| \jb{\nabla}^{\dl-\al}\U_N(t,x)\big|^{2p}\Big] 
= e^{p\be^2\s_N(t)}\E \Bigg[ \bigg| \int_{\T^2} 
J_{\al-\dl} (x-y) e^{i \be \Psi_N (t,y)} dy \bigg|^{2p} \Bigg]\\
    & =   e^{p\be^2\s_N(t)}\int_{(\T^{2})^{2p}} 
\E \bigg[  e^{i \be \sum_{j =1}^p  (\Psi_N (t,y_{2j}) - \Psi_N (t,y_{2j-1}))} \bigg] 
 \prod_{k=1}^{2p}  J_{\al-\dl} (x-y_k) d\vec y, 
\end{split}
\label{E0}
\end{align}
\noi
where $d\vec y \deff   dy_1\cdots d y_{2p}   $.
Noting that $ \sum_{j =1}^p  (\Psi_N (t,y_{2j}) - \Psi_N (t,y_{2j-1})) $ is a mean-zero Gaussian random variable, 
the explicit formula for the characteristic function of a Gaussian random variable yields
\begin{align}
\begin{split}
\E \bigg[   e^{i \be \sum_{j =1}^p  (\Psi_N (t,y_{2j}) - \Psi_N (t,y_{2j-1}))} \bigg]
     & = 
e^{- \frac{\be^2}2 \E \big[| \sum_{j =1}^p  (\Psi_N (t,y_{2j}) - \Psi_N (t,y_{2j-1})) |^2\big]}  .
\end{split}
\label{E1}
     \end{align}

 Let  $\{\eps_j\}_{j=1,...,2p}$ be as in Lemma~\ref{LEM:prod}.
Then,  we can rewrite the expectation  in the exponent on the right-hand side 
of \eqref{E1} as 
  \[\E\Big[\big|\sum_{j=1}^{2p}\eps_j\Psi_N(t,y_j)\big|^2\Big] 
  = \sum_{j,k=1}^{2p}\eps_j\eps_k \G_N(t,y_j-y_k), \]

\noi
where  $\G_N$ is the covariance function defined in \eqref{def-covar}. 
From the  definition \eqref{sig}, we have $\G_N(t,0)=\s_N(t)$.
Hence, we obtain
  \begin{equation} 
 e^{- \frac{\be^2}2 \E \big[| \sum_{j =1}^{p}  (\Psi_N (t,y_{2j}) - \Psi_N (y_{2j-1},t))|^2\big]} = e^{- p \be^2 \sigma_N(t)} e^{ - \be^2 \sum_{1\le j < k\le 2p} \eps_j\eps_k \G_N(t,y_j-y_k)  } .
\label{E2}
\end{equation}

 \noi
 Then, from \eqref{E1}, \eqref{E2}, 
the two-sided bound \eqref{bd-GN} in Lemma~\ref{LEM:covar}, 
 and  Lemma~\ref{LEM:prod}, we obtain
 \begin{equation}\label{E3}
 \begin{split}
 e^{p\be^2\s_N(t)} \E \bigg[  & e^{i \be \sum_{j =1}^p  (\Psi_N (t,y_{2j}) - \Psi_N (t,y_{2j-1}))} \bigg]\\
  & \sim  \prod_{1\le j < k \le 2p } \big(|y_j - y_k| + N^{-1}\big)^{\eps_j\eps_k \frac{\be^2 t}{4 \pi}}\\
 &\les \max_{\tau \in S_p} \prod_{1\leq j \leq p} \big(|y_{2j} - y_{2\tau(j)-1}| + N^{-1}\big)^{ -\frac{\be^2 t}{4 \pi}}\\
&\le \sum_{\tau \in S_p} \prod_{1\leq j \leq p} \big(|y_{2j} - y_{2\tau(j)-1}| + N^{-1}\big)^{ -\frac{\be^2 t}{4 \pi}}.
 \end{split}
 \end{equation}

 \noi
Finally,  from \eqref{E0} and \eqref{E3}
we obtain
 \begin{align}
\begin{split}
\E & \Big[\big| \jb{\nabla}^{\dl-\al}\U_N(t,x)\big|^{2p}\Big] 
 \\
 & \les \sum_{\tau \in S_p}\int_{(\T^{2})^{2p}}   
  \prod_{1\leq j \leq p} \big(|y_{2j} - y_{2\tau(j)-1}| + N^{-1}\big)^{ -\frac{\be^2 t}{4 \pi}} 
  \prod_{k=1}^{2p}  |J_{\al-\dl} (x-y_k)| d\vec y.
\end{split}
\label{E3a}
 \end{align}

 In the following, we fix $\tau \in S_p$.
 Then,  it suffices to bound each pair of integrals:
 \[\int_{\T^2}\int_{\T^2}\big(|y_{j} - y_{k}| + N^{-1}\big)^{ -\frac{\be^2 t}{4 \pi}} 
 | J_{\al-\dl} (x-y_{j})| |J_{\al-\dl} (x-y_{k})| dy_{j}dy_{k},\]
 
 \noi
 for  an even integer $j=2,...,2p$ and $k=2\tau(\tfrac{j}2)-1$.
  From Lemma~\ref{LEM:Bessel} with $0 < \al - \dl < 2$, 
  we can bound this integral by 
 \begin{align}
 \begin{split}
 \int_{\T^2} & \int_{\T^2}\big(|y_{j} - y_{k}|+N^{-1}\big)^{ -\frac{\be^2 t}{4 \pi}} |x-y_{j}|^{\al-\dl-2}
  |x-y_{k}|^{\al-\dl-2} dy_{j}dy_{k}\\
&  \leq 
 \int_{\T^2}\int_{\T^2}|y_{j} - y_{k}|^{ -\frac{\be^2 t}{4 \pi}} |x-y_j|^{\al-\dl-2}|x-y_k|^{\al-\dl-2} dy_{j}dy_k, 
\end{split}
\label{E4}
  \end{align}
 uniformly in $N\in\N$.

\smallskip
\noi
$\bullet$
{\bf Case 1:} 
 $|y_j-y_k|\sim |x-y_k|\gtrsim |x-y_j|$. 
\quad 
In this case, we bound the  integral in \eqref{E4} by 
\begin{align*}
\text{RHS of \eqref{E4}}
& \les \int_{\T^2} |x-y_k|^{\al-\dl-2-\frac{\be^2t}{4\pi}} 
\int_{|x-y_j|\les |x-y_k|}
|x-  y_j|^{\al-\dl-2}
dy_jdy_k\\ &\les \int_{\T^2}|x-y_k|^{2(\al-\dl) - 2 -\frac{\be^2t}{4\pi}}dy_k
\les 1,
\end{align*}

\noi
where we used $\frac{\beta^2t}{8\pi}<\al-\dl$.
The symmetry allows us to handle the case
$|y_j-y_k|\sim |x-y_j|\gtrsim |x-y_k|$.

 \smallskip
\noi
$\bullet$
{\bf Case 2:} 
$|x-y_j|\sim |x-y_k|\gtrsim |y_j-y_k|$.
\quad In this case,  we bound the  integral in \eqref{E4} by 
\begin{align*}
\text{RHS of \eqref{E4}}
& \les \int_{\T^2}|x-  y_k  |^{2(\al-\dl)-4} \int_{|y_j-y_k|\les |x-y_k|}|y_j-y_k|^{-\frac{\be^2t}{4\pi}}dy_jdy_k\\
& \les \int_{\T^2}|x-y_k|^{2(\al-\dl)-2-\frac{\beta^2t}{4\pi}}dy_k \les 1
\end{align*}

\noi
since $\frac{\beta^2t}{8\pi} < \min(\al - \dl, 1)$.

\smallskip

Putting together \eqref{E3a} and the estimates on \eqref{E4}, 
we obtain
\begin{equation}\label{E5}
\E \big[\big| \jb{\nabla}^{\dl-\al}\U_N(t,x)\big|^{2p}\big]\les 1,
\end{equation}

\noi
uniformly in $t\in[0,T]$, $x\in\T^2$, and $N\in\N$.
Therefore, we conclude from  \eqref{E0a} and \eqref{E5} that 
\[\|\U_N\|_{L^{2p}_\o L^q_TW^{-\al,\infty}_x}\les T^{\frac1q},\]
uniformly in $N\in\N$.

\smallskip

Next, we establish  convergence of $\U_N$.
Let $N_2 \geq N_1 \geq1$.
By 
repeating the  computation as above when $p=1$ with $\U_{N_1}-\U_{N_2}$ in place of $\U_N$, 
we have 
\begin{align}
\begin{split}
&\E \Big[ \big| \jb{\nabla}^{\dl-\al}\big(\U_{N_1}(t,x)-\U_{N_2}(t,x)\big) \big|^{2} \Big]
\\ &\quad =  \int_{\T^2}\int_{\T^2} J_{\al-\dl} (x-y)J_{\al-\dl} (x-z) \\
&\quad\quad \times 
\E\bigg[\Big(e^{ \frac{\be^2}2 \sigma_{N_1}(t)}e^{i \be \Psi_{N_1} (t,y)}-e^{ \frac{\be^2}2 \sigma_{N_2}(t)}e^{i \be \Psi_{N_2} (t,y)}\Big)\\
&\quad\quad\times \Big(e^{ \frac{\be^2}2 \sigma_{N_1}(t)}e^{-i \be \Psi_{N_1} (t,z)}-e^{ \frac{\be^2}2 \sigma_{N_2}(t)}e^{-i \be \Psi_{N_2} (t,z)}\Big)\bigg] dydz \\
&\quad = \sum_{j = 1}^2 
\int_{\T^2}\int_{\T^2} J_{\al-\dl} (x-y)J_{\al-\dl} (x-z) 
\Big(e^{ \be^2 \G_{N_j}(t,y-z) }-e^{ \be^2 \P_{N_1}\P_{N_2}\G(t,y-z)}\Big)dydz,
\end{split}
\label{H1}
\end{align}

\noi
where we used 
$\P_{N_1}\P_{N_2}\G(t,y-z) = 
\E\big[\Psi_{N_1}(t,y) \Psi_{N_2} (t,z)\big]$.
By  the mean value theorem
and  
the bounds \eqref{bd-GN} in Lemma \ref{LEM:covar}
and \eqref{G6}, we have 
\begin{align}
\begin{split}
& \Big|e^{ \be^2 \G_{N_j}(t,x) }-e^{ \be^2 \P_{N_1}\P_{N_2} \G(t,x)}\Big|\\
&\quad = \bigg|\big(\P_{N_j}^2\G(t,x) - \P_{N_1}\P_{N_2}\G(t,x)\big)\\
& \quad \quad \times \int_0^1\be^2 \exp\big[\be^2\big(\tau \P_{N_j}^2\G(t,x) +(1-\tau)\P_{N_1} \P_{N_2}\G(t,x)\big)\big]d\tau\bigg|\\
&\quad \les \big(|x|+N_2^{-1}\big)^{-\frac{\be^2 t}{4\pi}}
\big|\P_{N_j}^2\G(t,x) - \P_{N_1}\P_{N_2}\G(t,x)\big|.
\end{split}
\label{H2}
\end{align}

Given $\eps > 0$, there exists $C_\eps > 0$ such that 
\begin{align}
| \log y | \le C_\eps  y^{-\eps}
\label{H3}
\end{align}

\noi
 for any $0 < y \les 1$.
From \eqref{A6a} and \eqref{G5} along with \eqref{H3}, 
we have
\begin{align}
\begin{split}
\big|\P_{N_j}^2\G(t, x)-\P_{N_1} \P_{N_2}\G(t, x)\big| 
& \les T \big(|x|+N_2^{-1}\big)^{-\eps}\wedge\big(N_1^{-1}|x|^{-1}\big)\\
& \les T N_1^{-\eps} |x|^{-2\eps}
\end{split}
\label{H4}
\end{align}

\noi
for any $t \in[0,  T]$ and non-zero $x\in \T^2 \cong [-\pi,\pi)^2$.
Hence, from \eqref{H1}, \eqref{H2}, and \eqref{H4}
along with Lemma \ref{LEM:Bessel} (with $0 < \al - \dl < 2$), we obtain
\begin{align*}
&\E \Big[ \big| \jb{\nabla}^{\dl-\al}\big(\U_{N_1}(t,x)-\U_{N_2}(t,x)\big) \big|^{2} \Big]
\\ &\quad \les T \int_{\T^2}\int_{\T^2} |x-y|^{\al-\dl-2}|x-z|^{\al-\dl-2}N_1^{-\eps}|y-z|^{-2\eps-\frac{\be^2t}{4\pi}}dydz.
\end{align*}

\noi
By taking $\eps > 0$ sufficiently small such that 
 $\frac{\beta^2t}{8\pi} +\eps < \min(\al - \dl, 1)$, 
 we can proceed as in Cases 1 and 2 above and 
 estimate the integrals above by 
\begin{align}
\E \Big[ \big| \jb{\nabla}^{\dl-\al}\big(\U_{N_1}(t,x)-\U_{N_2}(t,x)\big) \big|^{2}\Big] \les T N_1^{-\eps}
\label{H5}
\end{align}

\noi
for any $N_2\geq N_1 \geq 1$,  $t\in [0,T]$, and $x\in\T^2$.

Given $p \geq 1$, we can interpolate the previous bound \eqref{H5} with \eqref{E5} 
and obtain 
\[\E \Big[ \big| \jb{\nabla}^{\dl-\al}\big(\U_{N_1}(t,x)-\U_{N_2}(t,x)\big) \big|^{p}\Big] \le C(T) N_1^{-\eps}\]

\noi
for any $N_2\geq N_1 \geq 1$,  $t\in [0,T]$, and $x\in\T^2$.
Therefore, we conclude that 
\[\|\U_{N_1}-\U_{N_2}\|_{L^{p}_\o L^q_TW^{-\al,\infty}_x} \le C(T)N_1^{-\tfrac{\eps}{p}}.\]

\noi
This shows that 
$\U_N$ is a Cauchy sequence in $L^p(\O;L^q([0,T];W^{-\al,\infty}(\T^2)))$. 
This completes the proof of  Proposition~\ref{PROP:Ups}.
\end{proof}

 \section{Proofs of the main results}

In this  section, we 
present a proof of local well-posedness 
of the renormalized stochastic sine-Gordon equation 
(Theorem~\ref{THM:main}).
Given 
the regularity and convergence properties
for the 
 imaginary  Gaussian multiplicative chaos $\U_N$
(Proposition~\ref{PROP:Ups}), 
Theorem~\ref{THM:main} follows from a standard application of the Strichartz estimate (see \eqref{Stri} below)
as in the work \cite{GKO} on the stochastic nonlinear wave equation
with a polynomial nonlinearity.
For this purpose, we first go over the 
Strichartz estimate along with other useful lemmas
 in Subsection~\ref{SUBSEC:est}.
 We then present a proof of 
Theorem~\ref{THM:main}
 in Subsection~\ref{SUBSEC:main1}.
We conclude this paper by 
establishing a triviality result (Proposition~\ref{PROP:triviality})
for the (unrenormalized) stochastic sine-Gordon equation \eqref{SSG}
in Subsection \ref{SUBSEC:main2}.

\subsection{Strichartz estimates
and other useful lemmas}
\label{SUBSEC:est}

We begin by recalling the Strichartz estimate for the linear wave equation on $\T^2$. 
Given 
 $0<s<1$, we say that a pair $(q, r)$ of exponents (and a pair $(\wt q,\wt r)$, respectively)
 is  $s$-admissible (and dual $s$-admissible, respectively), 
  if $1\leq \wt q\leq 2\leq q\leq\infty$ and $1<\wt r\leq 2 \leq r<\infty$ and if they satisfy the 
  following scaling and admissibility conditions:
\[\frac1q + \frac2r = 1-s = \frac1{\wt q}+\frac2{\wt r} -2, 
\qquad  
\frac2q+\frac1r\leq\frac12,
\qquad \text{and}\qquad  \frac2{\wt q}+\frac1{\wt r}\geq \frac52.\]

\noi
Then, it follows from 
Lemma 3.3 in \cite{GKO}, which  studies the maximization problem
 for  $J_s = \min(\frac{r}{\wt r}, \frac{q}{\wt q})$, $0 < s < 1$, that 
given $0 < s < 1$, 
there exist an $s$-admissible pair $(q,r) = (q(s), r(s))$ and a dual $s$-admissible pair 
$(\wt q,\wt r) = (\wt q(s),\wt r(s))$ such that 
\begin{equation}\label{bd-pairs}
q>2\wt q \qquad \text{and}\qquad r>2\wt r  .
\end{equation}

Let  $0 < s < 1$.
In the remaining part of this paper, 
we fix  pairs $(q,r)$ and $(\wt q,\wt r)$ as above, satisfying \eqref{bd-pairs}.
Given 
 $0<T\leq 1$, we then define the Strichartz space:
\begin{equation}\label{X}
X^s(T)=C([0,T];H^s(\T^2))\cap C^1([0,T];H^{s-1}(\T^2))\cap L^q([0,T];L^r(\T^2))
\end{equation}
and its ``dual'' space:
\[N^s(T) = L^1([0,T];H^{s-1}(\T^2)) + L^{\wt q}([0,T];L^{\wt r}(\T^2)),\]

\noi
where 
$X^s(T)$- and $N^s(T)$-norms are given by 
\[ \| u \|_{X^s(T)}
= 
\|u\|_{C_TH^s_x} + \|\dt u \|_{C_TH^{s-1}_x}
+ \|u\|_{L^q_TL^r_x} \]

\noi
and
\[ \|u \|_{N^s(T)}
= \inf \Big\{
\|u_1\|_{L^1_TH^{s-1}_x} + \|u_2\|_{L^{\wt q}_TL^{\wt r}_x}:
u = u_1 + u_2
\Big\}.\]

\noi
Then, 
the solution $u$ to the following linear wave equation:
\[\begin{cases}
\dt^2u+(1-\Dl)u = f\\
(u,\dt u)\big|_{t=0}=(u_0,u_1)
\end{cases}\] 

\noi
on $[0,T]$ satisfies the following Strichartz estimate:
\begin{equation}\label{Stri}
\|u\|_{X^s(T)} \les \|(u_0,u_1)\|_{\H^s} + \|f\|_{N^s(T)}, 
\end{equation}

\noi
uniformly in $T \in [0, 1]$.
The Strichartz estimates on $\R^d$ have been studied extensively by many
mathematicians.  See \cite{GV, LS, KeelTao}
in the context of the wave equation.
For the Klein-Gordon equation under consideration, 
see \cite{KSV}.
Thanks to the finite speed of propagation, 
the estimate \eqref{Stri} on $\T^2$ follows from the corresponding
estimate on $\R^2$.

Next, we recall the following product estimates.  See Lemma 3.4 in \cite{GKO}.
\begin{lemma}\label{LEM:bilin}
 Let $0\le s \le 1$.

\smallskip

\noi
\textup{(i)} Suppose that 
 $1<p_j,q_j,r < \infty$, $\frac1{p_j} + \frac1{q_j}= \frac1r$, $j = 1, 2$. 
 Then, we have  
\begin{equation}  
\| \jb{\nb}^s (fg) \|_{L^r(\T^d)} 
\les \Big( \| f \|_{L^{p_1}(\T^d)} 
\| \jb{\nb}^s g \|_{L^{q_1}(\T^d)} + \| \jb{\nb}^s f \|_{L^{p_2}(\T^d)} 
\|  g \|_{L^{q_2}(\T^d)}\Big).
\label{bilinear+}
\end{equation}

\smallskip

\noi
\textup{(ii)} 
Suppose that  
 $1<p,q,r < \infty$ satisfy the scaling condition:
$\frac1p+\frac1q=\frac1r + \frac{s}d $.
Then, we have
\begin{align}
\big\| \jb{\nb}^{-s} (fg) \big\|_{L^r(\T^d)} \les \big\| \jb{\nb}^{-s} f \big\|_{L^p(\T^d) } 
\big\| \jb{\nb}^s g \big\|_{L^q(\T^d)}.  
\label{bilinear-}
\end{align}

\end{lemma} 

Lastly, we recall the fractional chain rule from \cite{Gatto}. 
\begin{lemma}
\label{LEM:FC}
Suppose that $F$ is a Lipschitz function with Lipschitz constant $L$.
Then, given  any $0 < s < 1$ and $1 < p < \infty$,  we have
\begin{equation}\label{FC}
\big\| |\nabla|^s F(u)\big\|_{L^p(\T^d)} \les L \big\||\nabla|^s u\big\|_{L^{p}(\T^d)}
\end{equation}

\noi
 for any $u\in C^{\infty}(\T^d)$.
\end{lemma}
\noi

The fractional chain rule on $\R^d$ was essentially proved in \cite{CW}.\footnote{As pointed out in \cite{Staffilani}, 
the proof in \cite{CW} needs
a small correction, which  yields the fractional chain  rule in a 
less general context.
See \cite{Kato, Staffilani, Taylor}.}
As for the estimate \eqref{FC} on $\T^d$, see~\cite{Gatto}.

\subsection{Local well-posedness}
\label{SUBSEC:main1}

In this subsection, we present a proof of Theorem \ref{THM:main}.
For simplicity, we assume that $0 < s < 1$, 
in which case we have $\s = s$
in the statement of Theorem \ref{THM:main}.
Fix  $\be \in \R\setminus \{0\}$
and  $(u_0,u_1)\in\H^s(\T^2)$.
The main idea is to apply the Da Prato-Debussche trick.
Namely, 
we first write the solution $u_N$ to \eqref{RSSGN}
as $u_N=\Psi_N + v_N$, where $\Psi_N$ is
the truncated stochastic convolution defined in \eqref{PsiN}.
Then, the residual term $v_N$
satisfies \eqref{vN}. 
By writing  \eqref{vN} in the Duhamel formulation, we have
\begin{equation}\label{Duhamel}
v_N(t) = \dt S(t)u_0 + S(t)u_1 -\int_0^tS(t-t')\Im\big(\U_Ne^{i\be v_N}\big)(t')dt',
\end{equation}

\noi
where $S(t)$ and  $\U_N$ are as in \eqref{S} and \eqref{Ups}, respectively.
In the following, we use $B_R$ to denote the ball in $X^s(T)$ of radius $R > 0$
centered at the origin.

Let $\Phi^N(v_N) = \Phi^N_{(u_0, u_1), \U_N}(v_N)$ denote the right-hand side of \eqref{Duhamel}.
Let $0 < T \leq 1$.
By  the Strichartz estimate \eqref{Stri}, we have
\begin{align}
\big\| \Phi^N(v_N) \big\|_ {X^s(T)} 
\les 
\|(u_0, u_1) \|_{\H^s}  +  \Big\| \Im\big( \U_N e^{i\be v_N}\big)  \Big\|_{N^s(T)}.
\label{Z1}
\end{align}

\noi
From (3.22) and (3.23) in \cite{GKO}, we have 
\begin{equation}\label{interpolation}
\|f\|_{L^{q_1}_TW^{\al,r_1}_x}\les \|f\|_{X^s(T)} \qquad \text{ and }\qquad \|f\|_{N^s(T)} \les \|f\|_{L^{\wt q_1}_TW^{-\al,\wt r_1}_x},
\end{equation}

\noi
where $0<\al<\min(s,1-s)$ and the pairs $(q_1,r_1)$ and $(\wt q_1,\wt r_1)$ are given by
\begin{align}
 \frac1{ q_1} = \frac{1-\al/s}{q} + \frac{\al/s}{\infty}, \qquad  \frac1{r_1} = \frac{1-\al/s}{r} + \frac{\al/s}{2}   
 \label{Z1a}
\end{align}
and 
\begin{align}
 \frac1{\wt q_1} = \frac{\al/(1-s)}{1} + \frac{1-\al/(1-s)}{\wt q}, \qquad \frac1{\wt r_1} = \frac{\al/(1-s)}{2} + \frac{1-\al/(1-s)}{\wt r}   .  
\label{Z1b}
 \end{align}

\noi
By applying \eqref{interpolation}, the product estimate \eqref{bilinear-}
 in Lemma~\ref{LEM:bilin} (with sufficiently small $\al > 0$),
 and Lemma~\ref{LEM:FC}, 
 there exists $\ta > 0$ such that 
\begin{align}
\begin{split}
\Big\| \Im\big( \U_N e^{i\be v_N}\big)  \Big\|_{N^s(T)}  
& \les \Big\| \Im\big( \U_N e^{i\be v_N}\big)  \Big\|_{L^{\wt q_1}_T W^{-\alpha, \wt r_1}_x}\\
& \les T^\ta \big\| \U_N \big\|_{L^{q_2}_T W^{-\alpha,\infty}_x} \|  e^{i\be v_N}  \|_{L^{\wt q_2}_T W^{\alpha,\wt r_2}_x}
\\
& \les T^\ta \big\| \U_N \big\|_{L^{q_2}_T W^{-\alpha,\infty}_x}
\Big(1+\be\|  v_N  \|_{L^{\wt q_2}_T W^{\alpha,\wt r_2}_x}\Big)
\end{split}
\label{Z2}
\end{align}

\noi
for any finite $\wt q_2 >\wt q_1$ and $\wt r_2> \wt r_1$ with $\frac{1}{\wt q_1} > \frac{1}{q_2}+\frac{1}{\wt q_2}$.

It follows from \eqref{Z1a} and \eqref{Z1b} that by taking $\al \to 0$, 
we can take 
$(q_1, r_1)$ (and $(\wt q_1, \wt r_1)$, respectively)
arbitrarily close to $(q, r)$ (and to $(\wt q, \wt r)$, respectively).
Moreover,  by taking $\al > 0$ sufficiently small, 
 we can also take $(\wt q_2,\wt r_2)$ arbitrarily close to $(\wt q_1,\wt r_1)$.
 Hence, from~\eqref{bd-pairs}, we can guarantee
 \begin{equation}\label{bd-pairs2}
 2\wt q_2 < q_1 \qquad \text{ and } \qquad 2\wt r_2< r_1
\end{equation}
 
 \noi
 by taking $\al > 0$ sufficiently small.
Therefore, from \eqref{Z1} and  \eqref{Z2} with \eqref{interpolation} and \eqref{bd-pairs2}, we obtain
\begin{equation}\label{Z3}
\big\| \Phi^N(v_N) \big\|_ {X^s(T)} \les 
\|(u_0, u_1)\|_{\H^s} +
T^\ta \big\| \U_N \big\|_{L^{q_2}_T W^{-\alpha,\infty}_x}\Big(1+\be\|  v_N  \|_{X^s(T)}\Big).
\end{equation}

Proceeding as in \eqref{Z1} and \eqref{Z2}, 
we have 
\begin{align}
\big\| \Phi^N(v_N)-\Phi^N(w_N) \big\|_ {X^s(T)} 
\les T^{\ta}\big\| \U_N \big\|_{L^{q_2}_{T} W^{-\alpha,\infty}_x} \|  e^{i\be v_N} -e^{i\be w_N} \|_{L^{\wt q_2}_T W^{\alpha,\wt r_2}_x}.
\label{Z4}
\end{align}

\noi
As for the last factor in \eqref{Z4}, 
by applying  the mean value theorem
with $F(u)= e^{i\be u}$,
 the fractional Leibniz rule \eqref{bilinear+} in Lemma~\ref{LEM:bilin}, 
  and the fractional chain rule (Lemma~\ref{LEM:FC}), 
  we have
\begin{align}
\begin{split}
&\|  e^{i\be v_N} -e^{i\be w_N} \|_{L^{\wt q_2}_T W^{\alpha,\wt r_2}_x}\\ 
&\quad= \bigg\|  (v_N-w_N)\int_0^1F'(\tau v_N + (1-\tau)w_N)d\tau\bigg\|_{L^{\wt q_2}_T W^{\alpha,\wt r_2}_x} \\
&\quad\les \|  v_N-w_N\|_{L^{2\wt q_2}_T W^{\al,2\wt r_2}_x}\int_0^1\| F'(\tau v_N + (1-\tau)w_N)\|_{L^{2\wt q_2}_T W^{\alpha,2\wt r_2}_x}d\tau\\
&\quad\les \be\|  v_N-w_N\|_{L^{2\wt q_2}_T W^{\al,2\wt r_2}_x} 
\Big(1+\be\|  v_N\|_{L^{2\wt q_2}_T W^{\al,2\wt r_2}_x}+\be\|  w_N\|_{L^{2\wt q_2}_T W^{\al,2\wt r_2}_x}\Big).
\end{split}
\label{Z5}
\end{align}

\noi
Hence, from \eqref{Z4} and \eqref{Z5}, we obtain 
\begin{equation}\label{Z6}
\begin{split}
\big\| \Phi^N(v_N)-\Phi^N(w_N) \big\|_ {X^s(T)} 
&\les T^{\ta}\be(1+\be R)\big\| \U_N \big\|_{L^{ q_2}_{T} W^{-\alpha,\infty}_x}
\|  v_N-w_N\|_{X^s(T)}
\end{split}
\end{equation}

\noi
for $v_N, w_N \in B_R \subset X^s(T)$.

Let $\al > 0$ be sufficiently small as above.
Then, Proposition \ref{PROP:Ups} states that, 
given any finite $p\geq 1$,  
$\{\U_N\}_{N \in \N}$ is uniformly bounded in 
 $L^p(\O;L^{q_2}([0,T_0];W^{-\al,\infty}(\T^2)))$, 
  provided that 
\begin{align}
0 < T_0 < \frac{8\pi \al}{\be^2}.
\label{Z7}
\end{align}

\noi
Fix $T_0 > 0$ satisfying \eqref{Z7}.
Given $N \in \N$ and $\ld > 0$, define $\O_{N, \ld}$ by 
\[ \O_{N, \ld} = \Big\{\o \in \O: \big\| \U_N \big\|_{L^{ q_2}_{T_0} W^{-\alpha,\infty}_x}< \ld \Big\}.\]

\noi
On $\O_{N, \ld}$, 
it follows from \eqref{Z3} and \eqref{Z6} that 
$\Phi^N$
is a contraction on the ball $B_R\subset X^s(T)$, 
where $R\sim \| (u_0, u_1)\|_{\H^s}$
and $T = T(R, \ld, \be)
= T\big(\| (u_0, u_1)\|_{\H^s}, \ld, \be\big) > 0$
such that $T \leq T_0$.
Note that we can choose $T>0$ such that
\begin{align}
T \sim_{R, \be} (1+\ld)^{-\kk}
\label{Z7a}
\end{align}

\noi
for some $\kk > 0$ and any $\ld \gg 1$.

By Proposition \ref{PROP:Ups} and Chebyshev's inequality, 
we have 
\begin{align}
\sup_{N \in \N} P\big(  \O_{N, \ld}^c\big) = o(1)
\label{Z8}
\end{align}

\noi
as $\ld \to \infty$.
This proves  local well-posedness
of \eqref{Duhamel}
uniformly in $N \in \N$.
Here, the uniformity refers to the fact
that, given any small $T>0$,  we have a uniform (in $N$) control 
(i.e.~a uniform lower bound in terms of $T$) on 
the probabilities of the sets
$\O_N(T) = \O_{N, \ld(T)}$, where
well-posedness of \eqref{Duhamel} holds on $[0, T]$, 
thanks to \eqref{Z7a} and \eqref{Z8}.

Let $\U$ be the limit of $\U_N$ in 
 $L^p(\O;L^{q_2}([0,T_0];W^{-\al,\infty}(\T^2)))$
 constructed in Proposition~\ref{PROP:Ups}. 
Define  $\O_{\ld}$  by 
\begin{align}
 \O_{ \ld} = \Big\{\o \in \O: \| \U \|_{L^{ q_2}_{T_0} W^{-\alpha,\infty}_x}< \ld \Big\}.
\label{XY1}
 \end{align}

\noi
Then, by repeating the argument above, 
we see that 
the limiting equation \eqref{v1} for $v$, 
written in the Duhamel formulation as
\begin{equation}\label{Duhamel2}
v(t) = \dt S(t)u_0 + S(t)u_1 -\int_0^tS(t-t')\Im\big(\U e^{i\be v}\big)(t')dt', 
\end{equation}

\noi
is well-posed on the time interval $[0, T]$, 
where $T = T(R, \ld, \be) < T_0$
satisfies \eqref{Z7a}.
In view of Proposition~\ref{PROP:Ups}, 
we then conclude that 
there exists an almost surely positive stopping time 
\begin{align}
\tau
= \tau \big(\| (u_0, u_1)\|_{\H^s},  \be\big)
\sim C(\| (u_0, u_1)\|_{\H^s}, \be )
\big(1 + \| \U \|_{L^{ q_2}_{T_0} W^{-\alpha,\infty}_x}\big)^{-\kk}
\label{XY2}
\end{align}

\noi
 such that 
the limiting equation \eqref{v1} for $v$
is well-posed on the time interval $[0, \tau]$.

In the following, fix  sufficiently small $T = 
T\big(\| (u_0, u_1)\|_{\H^s},  \be\big) > 0$, satisfying \eqref{Z7}
(with $T_0$ replaced by $T$)
and
set  $\O(T) = \{ \tau \geq T\}$, 
where $\tau$ is defined in \eqref{XY2}.
Namely, for any $\o \in \O(T)$, 
the limiting equation~\eqref{Duhamel2} is well-posed on $[0, T]$
by  the argument above.
Moreover, 
in view of  \eqref{XY1} and \eqref{XY2}, 
we may assume that $\O(T) = \O_\ld$
for some $\ld = \ld(T) >0$ satisfying~\eqref{Z7a}.
Note that $P(\O(T)) > 0$ for any sufficiently small $ T>0$
and $P(\O(T)) \to 1$ as $T \to 0$.
In the following, we work with the conditional probability $P_T$
given the event $\O(T)$, 
defined by 
\[ P_T (A) = P(A|\O(T)) = \frac{P(A \cap \O(T))}{P(\O(T))}.\]

We first check that the truncated dynamics \eqref{Duhamel}
is well-posed on $[0, T]$
outside a set of $P_T$-probability $o(1)$ as $N \to \infty$.
Let $N \in \N$. By defining a set $\Si_N$ by 
\[ \Si_{N} = \Big\{\o \in \O: \big\| \U_N - \U \big\|_{L^{ q_2}_{T_0} W^{-\alpha,\infty}_x}\leq 1 \Big\},\]

\noi
we have 
\[
\O(T) \cap \Si_N =  \O_\ld \cap \Si_N \subset \O_{N, \ld+1} .\]

\noi
Without loss of generality,
 we may assume that
the equation \eqref{Duhamel}
is also well-posed on $[0, T]$
for any $\o \in \O(T) \cap \Si_N$.
Furthermore, since $\U_N$ converges in probability to $\U$
in  $L^{q_2}([0,T];W^{-\al,\infty}(\T^2))$, 
we have $P(\Si_N^c) \to 0 $ as $N \to \infty$.
Namely, we have
\begin{align}
P_T(\O(T) \cap \Si_N) \too 1, 
\label{Z9}
\end{align}

\noi
as $N \to \infty$.
This verifies well-posedness of the truncated dynamics \eqref{Duhamel}
on the time interval $[0, T]$
asymptotically $P_T$-almost surely
in the sense of \eqref{Z9}.

Given $(u_0, u_1) \in \H^s(\T^2)$ and $ \U_0 \in L^{q_2}([0,T];W^{-\al,\infty}(\T^2)))$,
define a map $\Phi = \Phi_{(u_0, u_1),  \U_0}$ by
\begin{equation}\label{Duhamelx}
\Phi(v)(t) = \dt S(t)u_0 + S(t)u_1 -\int_0^tS(t-t')\Im\big( \U_0 e^{i\be v}\big)(t')dt'.
\end{equation}

\noi
Then, a slight modification of the analysis presented above
shows that 
the map: 
\[\big((u_0, u_1) ,  \U_0\big) 
\in \H^s(\T^2)\times  L^{q_2}([0,T];W^{-\al,\infty}(\T^2)))
\too v \in X^s(T)\]

\noi
is continuous, where $v = \Phi(v)$ is the unique fixed point for $\Phi$
and $0 < T \leq T_0$ is sufficiently small.
From this observation
and convergence in probability of $\U_N$ to $\U$
deduced from Proposition \ref{PROP:Ups}, we conclude that 
the solution $v_N$ to \eqref{Duhamelx}
converges in probability with respect to the conditional probability $P_T$
to $v$ in $X^s(T)$.\footnote
{Note that the truncated equation  \eqref{Duhamel}
may not be well-posed on the time interval $[0, T]$, $P_T$-almost surely.
This, however, does not cause a problem since, as verified in \eqref{Z9},
the truncated equation \eqref{Duhamel} is well-posed on $[0, T]$
asymptotically $P_T$-almost surely.}

Finally, recalling the decompositions 
\[ u_N=\Psi_N + v_N 
\qquad \text{and} \qquad u = \Psi + v\]

\noi
and from  Lemma \ref{LEM:psi} that $\Psi_N$ converges $P$-almost surely 
(and hence $P_T$-almost surely)
to $\Psi$ in 
 $ C([0,T];W^{-\eps,\infty}(\T^2))$, 
we conclude that 
$u_N$ converges to $u$, 
in $ C([0,T];W^{-\eps,\infty}(\T^2))$, 
in probability with respect to the conditional probability $P_T$.
This completes the proof of Theorem \ref{THM:main}.

\subsection{Triviality of the unrenormalized model}
\label{SUBSEC:main2}

We conclude this paper by establishing a triviality result
for the unrenormalized model (Proposition \ref{PROP:triviality}),
whose proof follows from a small modification of the proof of Theorem \ref{THM:main}.

The main idea is to follow the idea from \cite{HRW, OPT1, OOR}
and artificially introduce a 
renormalization constant $\g_N(t)$ in \eqref{gN}.
Note that unlike the previous work  \cite{HRW, OPT1, OOR}, 
our renormalization constant appears in a multiplicative manner,
which makes our analysis particularly simple.
Start with the truncated equation \eqref{SSG2}
and rewrite it as
\[\dt^2u_N + (1-\Dl) u_N +\g_N^{-1} \Im\big(\g_N e^{i\be u_N}\big) = \P_N\xi.\]

\noi
With the decomposition
\begin{align}
u_N = \Psi_N + v_N,
\label{Z11a} 
\end{align}

\noi
it follows from \eqref{Ups} and \eqref{gN} that 
$\g_N e^{i\be u_N} = \U_N e^{i \be v_N}$.
Then, 
we may repeat the analysis presented in the previous subsection
and establish local well-posedness for 
 the equation satisfied by $v_N$:
\begin{equation}\label{uvN}
\begin{cases}
\dt^2v_N + (1-\Dl)v_N + \g_N^{-1} \Im\big(\U_N e^{i\be v_N}\big)=0\\
(v_N,\dt v_N)\big|_{t=0}=(u_0,u_1).
\end{cases}
\end{equation}

\noi
Noting that $0 \leq \g_N^{-1}(t) \leq 1$, 
we once again have 
 a uniform (in $N$) control 
on 
the probabilities of the sets
$\O_N(T) = \O_{N, \ld(T)}$, where
well-posedness of~\eqref{uvN} holds on $[0, T]$.
Moreover, we have
\begin{align}
\|  v_N  \|_{X^s(T)} \les \|(u_0, u_1)\|_{\H^s}
\label{Z10}
\end{align}

\noi
on $\O_N(T)$.

By  writing \eqref{uvN} in the  Duhamel formulation:
\[v_N(t)= \dt S(t)u_0 + S(t)u_1 -\int_0^tS(t-t')\g_N^{-1}(t')\Im\big(\U_Ne^{i\be v_N}\big)(t')dt',\]

\noi
we see from \eqref{Z1} and  \eqref{Z2} with \eqref{Z10} that 
\begin{align}
\begin{split}
\|v_N-\dt S(t)u_0 - S(t)u_1\|_{X^s(T)} 
&\les T^\ta \big\|\g_N^{-1}(t) \U_N \big\|_{L^{q_2}_T W^{-\alpha,\infty}_x}\Big(1+\be\|  v_N  \|_{X^s(T)}\Big)\\
&\les \|\g_N^{-1}(t) \|_{L^{\tfrac1{\dl}}_T}\big\|\U_N \big\|_{L^{\frac{q_2}{1-\dl q_2}}_{T_0} W^{-\alpha,\infty}_x}\\
& \quad 
\times \Big(1+\be \| (u_0,u_1)\|_{\H^s}\Big)
\end{split}
\label{Z11}
\end{align}

\noi
for any small $\dl>0$.
From \eqref{gN} with \eqref{sig} and  Lemma~\ref{LEM:covar}, we have
\begin{align}
\|\g_N^{-1}(t) \|_{L^{\tfrac1{\dl}}_T} \les \| e^{-\frac{\be^2t}{8\pi}\log N}\|_{L^{\tfrac1{\dl}}_T}\leq C(\be,\dl)\big(\log N\big)^{-1} \too 0,
\label{Z12}
\end{align}

\noi
as $N \to \infty$.
Then, it follows
from 
\eqref{Z11}, \eqref{Z12}, 
 Proposition \ref{PROP:Ups},
 and Chebyshev's inequality
 that 
  $v_N -\dt S(t) u_0 - S(t)u_1$ converges in probability to 0 in $X^s(T)$ as $N \to \infty$.
  Recalling the decomposition \eqref{Z11a}
  and the almost sure convergence of  $\Psi_N$ to  $\Psi$ 
in    $ C([0,T];W^{-\eps,\infty}(\T^2))$
  (Lemma~\ref{LEM:psi}), 
we conclude that $u_N$ converges 
 in probability to $ \dt S(t)u_0 + S(t) u_1 + \Psi$, which is the unique solution to the linear stochastic wave equation \eqref{LSW}
 with initial data $(u_0, u_1)$.
This completes the proof of  Proposition \ref{PROP:triviality}.

\begin{acknowledgment}

\rm 
T.O.~and T.R.~were supported by the European Research Council (grant no.~637995 ``ProbDynDispEq'').
P.S.~was partially supported by NSF grant DMS-1811093.
The authors would like to thank the anonymous referees for the helpful comments.

\end{acknowledgment}


\begin{thebibliography}{99}





\bibitem{AHR}
S.~Albeverio, Z.~Haba, F.~Russo, 
{\it Trivial solutions for a nonlinear two-space-dimensional wave equation perturbed by space-time white noise,} Stochastics Stochastics Rep.  56 (1996), no.~1-2, 127--160. 


\bibitem{AHR2}
S.~Albeverio, Z.~Haba, F.~Russo, 
{\it A two-space dimensional semilinear heat equation perturbed by (Gaussian) white noise,}
Probab. Theory Related Fields 121 (2001), no. 3, 319--366. 


\bibitem{AS}
N.~Aronszajn, K.~Smith, 
{\it Theory of Bessel potentials. I,}
 Ann. Inst. Fourier (Grenoble) 11 (1961),  385--475. 


\bibitem{BEMS}
A.~Barone, F.~Esposito, C.~Magee, A.~Scott, 
{\it Theory and applications of the sine-Gordon equation},
 Rivista del
Nuovo Cimento, 
1 (1971),  227--267.

\bibitem{BO}
\'A.~B\'enyi,
T.~Oh, 
{\it  The Sobolev inequality on the torus revisited}, Publ. Math. Debrecen 83 (2013), no. 3, 359--374. 

\bibitem{BO94} J.~Bourgain, {\it Periodic nonlinear Schr\"odinger equation and invariant measures}, Comm. Math. Phys. 166 (1994), 1--26.

\bibitem{BO96} J.~Bourgain,  {\it Invariant measures for the 2D-defocusing nonlinear Schr\"odinger equation},  Comm. Math. Phys. 176 (1996), 421--445. 

\bibitem{CHS}
A.~Chandra, M.~Hairer, H.~Shen, 
{\it The dynamical sine-Gordon model in the full subcritical regime}, arXiv:1808.02594 [math.PR].



\bibitem{CW}
M.~Christ, M.~Weinstein,
{\it Dispersion of small amplitude solutions of the generalized Korteweg-de Vries equation.}
J. Funct. Anal.
100 (1991), 87--109.





\bibitem{DPD}
G.~Da Prato, A.~Debussche, 
{\it Strong solutions to the stochastic quantization equations,} Ann. Probab. 31 (2003), no. 4, 1900--1916.




\bibitem{DPZ14}
G.~Da~Prato, J.~Zabczyk,
\textit{Stochastic equations in infinite dimensions},
Second edition. Encyclopedia of Mathematics and its Applications, 152. Cambridge University Press, Cambridge, 2014. xviii+493 pp.



\bibitem{EJS}
W.~E, A.~Jentzen, H.~Shen, 
{\it Renormalized powers of Ornstein-Uhlenbeck processes and well-posedness of stochastic Ginzburg-Landau equations,} 
Nonlinear Anal. 142 (2016), 152--193.

\bibitem{Fro}
J.~Fr\"ohlich, 
{\it Classical and quantum statistical mechanics in one and two dimensions: two-component Yukawa- and Coulomb systems.}
Comm. Math. Phys. 47 (1976), no. 3, 233--268.



\bibitem{Ga}
C.~Garban,
{\it Dynamical Liouville},
arXiv:1805.04507 [math.PR].

\bibitem{Gatto}
A.E.~Gatto,
{\it Product rule and chain rule estimates for fractional derivatives on spaces that satisfy the doubling condition}, J. Funct. Anal. 188 (2002), no. 1, 27--37. 







\bibitem{GV}
J.~Ginibre, G.~Velo, 
{\it Generalized Strichartz inequalities for the wave equation,} J. Funct. Anal. 133 (1995),
50--68.



\bibitem{Gra1}
L.~Grafakos,
{\it Classical Fourier analysis},
Third edition. Graduate Texts in Mathematics, 249. Springer, New York, 2014. xviii+638 pp.

\bibitem{Gra2}
L.~Grafakos,
{\it Modern Fourier analysis},
Third edition. Graduate Texts in Mathematics, 250. Springer, New York, 2014. xvi+624 pp. 


\bibitem{GKO}
M.~Gubinelli, H.~Koch, T.~Oh,
{\it Renormalization of the two-dimensional stochastic nonlinear wave equation}, Trans. Amer. Math. Soc. 370 (2018), 7335--7359.

\bibitem{GKO2}
M.~Gubinelli, H.~Koch, T.~Oh,
{\it Paracontrolled approach to the three-dimensional stochastic nonlinear wave equation with quadratic nonlinearity},
arXiv:1811.07808 [math.AP].

\bibitem{GKOT}
M.~Gubinelli, H.~Koch, T.~Oh, L.~Tolomeo,
{\it Global dynamics for  the two-dimensional stochastic nonlinear wave equations,}
preprint.



\bibitem{H}
M.~Hairer,
{\it A theory of regularity structures},
Invent. math. (2014) 198: 269--504.

\bibitem{HRW}
M.~Hairer, M.~D.~Ryser, H.~Weber,
{\it Triviality of the 2D stochastic Allen-Cahn equation},
Electron. J. Probab. 17 (2012), no. 39, 14 pp. 

\bibitem{HS}
M.~Hairer, H.~Shen,
{\it The dynamical sine-Gordon model}, Comm. Math. Phys. 341 (2016), no. 3, 933--989. 


\bibitem{Kato}
T.~Kato,
{\it On nonlinear Schr\"odinger equations. II. $H^s$-solutions and unconditional well-posedness},
J. Anal. Math. 67 (1995), 281--306. 


\bibitem{KeelTao}
M.~Keel, T.~Tao, {\it Endpoint Strichartz estimates},  Amer. J. Math. 120 (1998), no. 5, 955--980.


\bibitem{KSV}
R.~Killip, B.~Stovall, M.~Visan, 
{\it Blowup behaviour for the nonlinear Klein-Gordon equation,} Math. Ann. 358 (2014), no. 1-2, 289--350.



\bibitem{LRV}
H.~Lacoin, R.~Rhodes, V.~Vargas,
{\it Complex Gaussian multiplicative chaos}, 
Comm. Math. Phys. 337 (2015), no. 2, 569--632. 

\bibitem{LRV2}
H.~Lacoin, R.~Rhodes, V.~Vargas,
{\it A probabilistic approach of ultraviolet renormalisation in the boundary Sine-Gordon model},
arXiv:1903.01394 [math.PR].


\bibitem{LS}
H.~Lindblad, C.~Sogge, {\it On existence and scattering with minimal regularity for semilinear wave equations,} J. Funct. Anal. 130 (1995), 357--426.



\bibitem{McKean81}
H.P.~McKean,
{\it The sine-Gordon and sinh-Gordon equations on the circle},
Comm. Pure Appl. Math. 34 (1981), no. 2, 197--257.

\bibitem{McKean94}
H.~P.~McKean, K.~L.~Vaninsky,
{\it Statistical mechanics of nonlinear wave equations}, Trends and perspectives in applied mathematics, 239--264,
Appl. Math. Sci., 100, Springer, New York, 1994. 

\bibitem{OOR}
T.~Oh, M.~Okamoto, T.~Robert,
{\it A remark on triviality for the two-dimensional stochastic nonlinear wave equation},
arXiv:1905.06278 [math.AP].


\bibitem{OOTz}
T.~Oh, M.~Okamoto, N.~Tzvetkov,
{\it  Uniqueness and non-uniqueness of the Gaussian free field evolution under the two-dimensional Wick ordered cubic wave equation}, preprint.


\bibitem{OPT1} T.~Oh, O.~Pocovnicu, N.~Tzvetkov, 
{\it Probabilistic local well-posedness of the cubic nonlinear wave equation in negative Sobolev spaces},
arXiv:1904.06792 [math.AP].

\bibitem{ORSW2}
T.~Oh, T.~Robert, P.~Sosoe,  Y.~Wang, 
{\it Invariant Gibbs dynamics for the dynamical sine-Gordon model}, 
preprint.


\bibitem{ORTz}
T.~Oh, T.~Robert, N.~Tzvetkov,
{\it Stochastic nonlinear wave dynamics on compact surfaces}, 
arXiv:1904.05277 [math.AP].


\bibitem{ORW}
T.~Oh, T.~Robert, Y.~Wang,
{\it  On the parabolic and hyperbolic Liouville equations},
arXiv:1908.03944 [math.AP].


\bibitem{OTh2} T.~Oh, L.~Thomann, {\it  Invariant Gibbs measures for the 2-$d$ defocusing nonlinear wave equations}, 
to appear in Ann. Fac. Sci. Toulouse Math.


\bibitem{PS}
J.~Perring, T.~Skyrme, 
{\it A model unified field equation}, 
Nuclear Phys. 31 (1962) 550--555. 

\bibitem{RSS}
S.~Ryang, T.~Saito, K.~Shigemoto,
{\it Canonical stochastic quantization},
Progr. Theoret. Phys. 73 (1985), no. 5, 1295--1298. 


\bibitem{Staffilani}
G.~Staffilani,
{\it The initial value problem for some dispersive differential equations},
Thesis (Ph.D.) -- The University of Chicago. 1995. 88 pp.


\bibitem{ST}
C.~Sun, N.~Tzvetkov,
{\it New examples of probabilistic well-posedness for nonlinear wave equations},
arXiv:1903.04441 [math.AP].

\bibitem{Taylor}
M.~Taylor,
{\it Tools for PDE}, 
Pseudodifferential operators, paradifferential operators, and layer potentials. Mathematical Surveys and Monographs, 81. American Mathematical Society, Providence, RI, 2000. x+257 pp.


\end{thebibliography}
\end{document}